\documentclass[reqno, 12pt]{amsart}

\numberwithin{equation}{section}

\usepackage{amssymb}
\usepackage{mathrsfs}
\usepackage{dsfont}
\usepackage{stmaryrd, MnSymbol}
\usepackage{enumerate, xspace}
\usepackage{changebar}
\usepackage{hyperref}
\usepackage{pdfsync}

\newtheorem{thmm}{Theorem}
\newtheorem{ass}{Assumption}
\newtheorem{thm}{Theorem}[section]
\newtheorem{lem}[thm]{Lemma}
\newtheorem{cor}[thm]{Corollary}

\newtheorem{prop}[thm]{Proposition}

\newtheorem{rem}[thm]{Remark}

\newcommand\cC{{\mathcal C}}
\newcommand\cD{{\mathcal D}}
\newcommand\cE{{\mathcal E}}

\newcommand\cH{{\mathcal H}}

\newcommand\cO{{\mathcal O}}
\newcommand\cM{{\mathcal M}}
\newcommand\cN{{\mathcal N}}

\newcommand\cZ{{\mathcal Z}}

\newcommand\bE{{\mathbb E}}

\newcommand\bR{{\mathbb R}}

\newcommand\bZ{{\mathbb Z}}

\newcommand\scH{{\mathscr H}}

\newcommand\ve{\varepsilon}

\newcommand\vf{\varphi}

\newcommand{\mcE}{{\mc E\!\!\!\!\mc E\!\!\!\!\mc E}}
\newcommand{\mcEe}{{\mc E\!\!\!\!\mc E\!\!\!\!\mc E}^{\,\ve}}

\newcommand\oU{\overline{U}}
\newcommand\Id{{\mathds{1}}}

\newcommand{\bigtime}{{\sf X}}

\newcommand{\lv}{\llangle}
\newcommand{\rv}{\rrangle}

\newcommand{\mc}[1]{{\mathcal #1}}

\begin{document}

\title[Fourier law]{Toward the Fourier law for a weakly interacting anharmonic crystal}
\author{Carlangelo Liverani}
\address{Carlangelo Liverani\\
Dipartimento di Matematica\\
II Universit\`{a} di Roma (Tor Vergata)\\
Via della Ricerca Scientifica, 00133 Roma, Italy.}
\email{{\tt liverani@mat.uniroma2.it}}
\author{Stefano Olla}
\address{Stefano Olla\\
CEREMADE, UMR CNRS 7534\\
Universit\'e Paris-Dauphine\\
75775 Paris-Cedex 16, France, \emph{and}}
\address{
 INRIA - Universit\'e Paris Est, CERMICS, Projet MICMAC, Ecole des
 Ponts ParisTech\\ 
  6 \& 8 Av. Pascal, 77455 Marne-la-Vall\'ee Cedex 2, France }
\email{{\tt olla@ceremade.dauphine.fr}}
\date{\today.}
\begin{abstract}
For a system of weakly interacting anharmonic oscillators, perturbed
by an energy preserving stochastic dynamics, we prove an autonomous
(stochastic) evolution for the energies at large time scale (with
respect to the coupling parameter). It turn out that this macroscopic
evolution is given by the so called conservative (non-gradient)
Ginzburg-Landau system of stochastic differential equations.
The proof exploits hypocoercivity and hypoellipticity properties of
the uncoupled dynamics.  
\end{abstract}
\thanks{It is a pleasure to thank  Clement Mouhot for many very useful
  discussion. We are also indebted to  Boguslaw Zegarli\'nski for
  helpful suggestions. This paper has been partially supported by the
  European Advanced Grant {\em Macroscopic Laws and Dynamical Systems}
  (MALADY) (ERC AdG 246953), by Agence Nationale de la Recherche, 
  under grant ANR-07-BLAN-2-184264 (LHMSHE) and by MIUR under the grant PRIN 2007B3RBEY} 
\maketitle


\section{Introduction}
\label{sec:intro}

The problem of deriving heat equation and Fourier's law for the
macroscopic evolution of the energy from a microscopic dynamics of
interacting atoms (hamiltonian or quantum), is one of the major goals
of non-equilibrium statistical mechanics \cite{BLR}. 

Although we are still very far from a rigorous mathematical derivation,
we have now some understanding of the needed ingredients.

It is clear that heat equation is a \emph{macroscopic}
phenomenon, emerging after a diffusive rescaling of space and time. It is
also clear that non-linearities of the microscopic dynamics are
necessary, since in a linear system of interacting oscillators energy may
disperse ballistically and thermal conductivity results infinite
\cite{RLL}. 
Non linearities of the interaction should give enough chaoticity and
time mixing such that \emph{locally} the system, in the macroscopic
time scale, is in a state of \emph{local equilibrium}. This should be
intended in terms of a scale parameter $\ve$: in a region of linear
size $\ve$, at a large time scale $\ve^{-b}t$, the system should be
close to equilibrium with temperature given by the local average of
kinetic energy. This statement of \emph{local equilibrium} should be
intended in the stronger sense that locally the dynamics is close to
an equilibrium dynamics.
Since energy is a conserved quantity, it can only evolve by moving
between different regions of linear size $\ve^{-1}$ through energy
currents. Because of the size of the regions and the fact that in
equilibrium energy currents have null average, one should look at time
of the order $\ve^{-2}t$ in order to see some exchange of energy
between boxes at different temperature. In other words a
central limit theorem for the energy currents is involved, and the
thermal conductivity is then given by the space-time integral of the
current-current  correlation (Green-Kubo formula). This conductivity
will be convergent if the system in equilibrium has enough mixing
properties. 

 To perform the above program, in a mathematical rigorous way, from a
 purely deterministic Hamiltonian dynamics, it is at the moment a too
 difficult challenge.

In the last years some mathematical results have been obtained by
perturbing the dynamics with energy conserving stochastic forces. 
The purpose of this stochastic perturbations is to give the ergodic
and chaotic properties to the system without modifying the
macroscopic behavior of the evolution of the energy.

This strategy has proven successful for systems in the hyperbolic
scaling ($b=1$), when momentum conservation is also preserved by the
stochastic perturbations \cite{ovy}, obtaining Euler system of equations
for compressible gas as macroscopic equation, at least in the smooth
regime.

In the diffusive scaling ($b=2$) this problem is still very
challenging even in presence of the stochastic perturbations. 
The difficulty is essentially involved in the space-time rescaling and
the corresponding central limit theorem.

In this paper we develop a weak-coupling approach to the problem of
energy diffusion that permits to separate the time limit from the
space one. We consider a \emph{finite} system of anharmonic
oscillators, whose hamiltonian dynamics is perturbed by a noise that
conserves the kinetic energy of each oscillators (we consider
oscillators that have at least 2 degree of freedom). The noise could
be though as modeling some chaotic internal degree of freedom of each
atom.

The oscillators are \emph{weakly} coupled with a small parameter
$\ve$. Consequently the exchange of energy between oscillators is
given by the currents associated to the hamiltonian mechanism, but
multiplied by $\ve$.  The noise drives each atom towards
(microcanonical) equilibrium where currents have null average. In a
time scale of order $\ve^{-2}$ current fluctuations are able to move
energy around the lattice if gradients of energy are present between
atoms.

We prove in fact that in the limit as $\ve\to 0$, in the time scale
$\ve^{-2}t$, the energies of the atoms evolve autonomously following
the solution of a system of stochastic differential equations,
conservative of the total energy (cf. (\ref{eq:opla})).  
It turns out that this macroscopic stochastic evolution has already
been considered in the hydrodynamic limit literature and it is called
\emph{non-gradient} Ginzburg-Landau model \cite{Va}. Consequently,
using the techniques developed in \cite{Va}, one could try to prove that
under space-time diffusive rescaling, the energy evolves following a
non-linear heat equation. The results in \cite{Va} do not apply
directly to  (\ref{eq:opla}) due to degeneracy of the coefficients,
but hopefully they can be adapted to the present situation, for the moment we postpone this problem.
We have thus reduced the derivation of the heat equation to a two step
procedure of which this paper rigorously accomplish the first step. 

We should remark here that an extension of the non-gradient technique
of   \cite{Va} directly to our original microscopic stochastic
dynamics would be much more challenging, as this dynamics is very
degenerate. 

The reason of the name non-gradient comes from the fact that the
currents of the macroscopic dynamics are not gradient, i.e. are not a
given by the spacial gradient of a local function of the
configurations of energies. It is interesting to note that in the
purely harmonic case, the macroscopic dynamics (\ref{eq:opla}) become
gradient  (see appendix \ref{app:gauss}). This implies a connection
between the non-gradient property and non-linearity of the
microscopic dynamics. We also notice that in the purely harmonic
case, because the presence of the energy conserving noise, the
microscopic energy current have an exact fluctuation-dissipation
decomposition in a gradient  plus a fast fluctuating term (see formula
(\ref{eq:poisson} in appendix \ref{app:gauss}), that has been already
exploited in \cite{BO} to obtain Fourier's law.

The main strategy of the proof is similar to other averaging
principles (\cite{freidlin, FW, Ki}): at a large time scale the dynamics of
each atom is close to the equilibrium dynamics parametrized by its own
energy. Energies of the atoms are our \emph{slow variables} and
evolves through their currents. But a simple averaging of these
currents (that would occur in a time scale of order $\ve^{-1}$) would
not move any energy, since currents have null averages respect to all
equilibrium measures (microcanonical). This forces us to look at the time
scale $\ve^{-2}$, when the energy evolves due to the
\emph{fluctuations} of the currents in equilibrium. 
Thus we must establish a central limit theorem for the energy
currents in the uncoupled dynamics, i.e. we have to study the Poisson
equation
\begin{equation*}
  L_0 u = j
\end{equation*}
 where $L_0$ is the generator of the uncoupled dynamics, and $j$ is
 the energy current between two particles. In order to prove our
 theorem, we need existence and regularity of the solution $u$ of this
 equation. The generator $L_0$ turns out to be hypoelliptic on each
 microcanonical energy surface, that provides regularity on the
 tangent direction of this surface. Yet, as the energy is exchanged
 from one particle to the other, we also needed to prove regularity in
 non tangential directions.

For the existence of $u$, we prove a spectral gap in a proper Sobolev
space, with an adaptation of hypocoercivity techniques \cite{V1}.
These techniques provide a precise control of this spectral gap with respect to the
energy, a control especially needed at low energies. Such detailed informations are necessary in 
order to perform the closure of the macroscopic equations. 

The content of the paper is as follows.

\tableofcontents

\section{The model}
\label{sec:model}

Let us consider a region $\Lambda \subset \bZ^d$, set $N=|\Lambda|$,
the number of sites in $\Lambda$. At each site we have a
$\nu$-dimensional, $\nu\geq 2$, nonlinear oscillator and we assume that such
oscillators interact weakly via a non-linear potential. Such a
situation is described by the following Hamiltonian in the variables   
$(q_i,p_i)_{i\in\Lambda}\in \bR^{2\nu N}$
\[
H_\ve^\Lambda:=\sum_{i\in\Lambda}\frac
12\|p_i\|^2+\sum_{i\in\Lambda}U(q_i)+\ve\frac 1{2} \sum_{|i-j|=1} V(q_i-q_j),  
\]
where $U, V\in\cC^\infty(\bR^\nu,\bR)$. We use the convention
$\sum_{|i-j|=1} = \sum_{i\in\Lambda} \sum_{\{j\in\Lambda\;:\;|i-j|=1\}}$.

For simplicity we assume that  $U(q)=\oU(|q|^2)$,
$\oU\in\cC^\infty(\bR,\bR)$, $\oU(0)=0$ and $c^{-1}\leq \oU'\leq c$
for some finite positive constant $c$, in particular this implies that
$U$ is radially symmetric and strictly convex.\footnote{The general, non radial, case can be treated exactly in the same way at the only price of a  much messier algebra. On the contrary, the non convex case could hold interesting surprises and hopefully will be investigated in the future.} 
Also, we assume $\|\nabla V(q)\|^2\le c U(q)$ and $V(-q)=V(q)$.

For simplicity of notations, we choose $\nu = 2$. All result stated in
this paper are valid for general $\nu\ge 2$, with slight modifications
of notations.

In addition to the Hamiltonian dynamics, we consider random forces that conserve the single sites
kinetic energies, given by independent diffusions on the spheres $\|p_i\|^2=cost$. In order to
define such diffusions, consider the vector fields
\begin{equation*}
  X_{i} := p_i^1 \partial_{p_i^2} -  p_i^2 \partial_{p_i^1}=:J p_i\cdot 
  \partial_{p_i},
  \qquad J = \begin{pmatrix}0& -1 \\ 1&0\end{pmatrix}
\end{equation*}
and the second order operator
\begin{equation*}
  S = \sum_{i\in\Lambda} X_{i}^2
\end{equation*}

The generator of the process we are interested in is then given by
\begin{equation}\label{eq:generator}
L_{\ve,\Lambda}=: A_\ve + \sigma^2 S 
\end{equation}
where $A_\ve = \{H_\ve^\Lambda,\cdot\}$,  is the usual Hamiltonian operator and $\sigma>0$ measures the strength of the noise. Clearly, $L_{\ve,\Lambda}$ is the generator of a contraction semigroup $P_\ve^t$ in $L^2(\bR^{2N\nu},m_\ve)$ with stationary measure $m_\ve$ and $P_\ve^tH_\ve^\Lambda=H_\ve^\Lambda$, for all $t\in\bR_+$. Next, we must specify the initial conditions.

The  Gibbs measures at temperature $\beta^{-1}$ are defined by
 $$
m^\beta_\ve(dq,dq)=Z_\ve(\beta) e^{-\beta H_\ve^\Lambda(q,p)}dpdq
$$
and are the stationary (equilibrium) probability measures for the dynamics (the canonical ensemble). 
As reference measure we pick the one corresponding to $\beta =1$ and
we denote it by $m_\ve$. Notice that for $\ve$ small enough, $m_\ve$
and the product measure $m_0$ are equivalent.  
To simplify notations, we also assume that $U$ is such that 
$$
Z_0(1) = 1 =\int_{\bR^4} e^{-(p^2/2 + U(q))}dp dq.
$$

We assume that the system is started in an initial distribution 
$d\nu_0:= F_\ve dm_\ve = F_0 dm_0$.
\begin{ass}\label{ass:zero}
We assume that $F_0\in L^2(\bR^{4N},m_0)$.
\end{ass}
For each $T>0$, the Markov process just described defines a
probability on $\Omega=\cC^0([0,T], \bR^{2\nu |\Lambda|})$. We will
use $\omega_t=(q(t),p(t))$ to designate the elements of $\Omega$ at
time $t$.

\begin{rem}
In the following we will suppress the subscripts and superscripts
$\Lambda$, when this does not create confusion. 
\end{rem}
\begin{rem}
Here we have free boundary conditions. It should be possible to treat
more general stochastic boundary conditions (e.g. having the particle
at the boundary perform an Ornstein-Uhlenbeck process at a given
temperature) by the same method, we avoid such a generalization to
simplify the presentation.  
\end{rem}

\section{The results}\label{sec:res}
The single particles energies are 
$$
\mc E_i^\ve(q,p) = \frac
12\|p_i\|^2+U(q_i)+\frac 1{2}\ve\sum_{|i-j|=1}V(q_i-q_j) .
$$
The time evolution of these energies is given by:\begin{equation}\label{eq:energy}
\frac{d \mc E_i^\ve}{dt}=\ve\sum_{|i-k|=1}j_{i,k}\\
\end{equation}
where the energy currents are defined by
\begin{equation}\label{eq:j}
  j_{i,k} = \frac 1{2}\nabla V(q_i-q_k)\cdot (p_i+p_k). 
\end{equation}
Note that $j_{i,k}=-j_{k,i}$ and that they are functions of the
$q_i, p_i, q_k, p_k$ only.

If $\ve = 0$ the dynamics is given by non-interacting oscillators, and consequently the energy of each oscillator is a conserved quantity.
So for $\ve =0$ there is a family of equilibrium measure parametrized by the vector
 $\underline a = (a_i)_{i\in\Lambda}$ of the energy of each oscillator.
This is given by 
$\mu_{\underline a}^\Lambda$, the microcanonical measure associated to
the Hamiltonian flow $H_0^{\Lambda}$ on the surface 
\begin{equation}\label{eq:energy-surf}
\Sigma_{\underline a}:=\left\{q,p\;:\; a_i = \mc E_i^0(q,p) = \frac
12\|p_i\|^2+U(q_i)\right\} = \bigtime_{i\in\Lambda} \Sigma_{a_i} .
\end{equation}
Clearly, letting $\mu_a$ be the microcanonical
measure on the $3$ dimensional surface $\Sigma_a$, we
have $\mu_{\underline a}^\Lambda=\otimes_{i\in\Lambda}\mu_{a_i}$. 
By the symmetry between $p$ and $-p$ it follows that
$\mu_{\underline a}(j_{i,k})=0$ for each $\underline a$.

 We are interested in the random variables determined the 
time rescaled energies 
\begin{equation}
{\mc E\!\!\!\!\mc E\!\!\!\!\mc E}_i^{\,\ve}(t)=  
\mc E_i^\ve (q(\ve^{-2}t),p(\ve^{-2}t)).\label{eq:process}
\end{equation}

In order to define the parameters of the mesoscopic evolution,
consider the dynamics of 2 non-interacting oscillators ($\ve = 0$),
each starting with the microcanonical distribution with corresponding energy $a_1$ and $a_2$. Let us denote by $ \mathbb E_{a_1, a_2}(\cdot)$ the corresponding expectation in this equilibrium measure.
We will show that the following function on $\bR_+^2$
\begin{equation}
  \label{eq:13}
  \gamma^2(a_1, a_2) = \int_0^\infty \mathbb E_{a_1, a_2}
  \left(j_{1,2}(t) j_{1,2}(0) \right) \; dt,
\end{equation}
is well defined.  More, in Lemma \ref{lem:regular-ag} we prove that 
\begin{equation}
  \label{eq:1}
  \gamma^2(a_1, a_2) = a_1 a_2 G (a_1, a_2)
\end{equation}
where $G$ is a positive symmetric smooth function.
Correspondingly we define the mesoscopic current by the
antisymmetric function 
\begin{equation}\label{eq:ag}
\alpha(a_1,a_2) =  e^{\mathcal U(\underline a)} 
    (\partial_{a_1}- \partial_{a_2} )  
   \left( e^{- \mathcal U(\underline a)} \gamma^2(a_1,a_2)\right). 
\end{equation}
where 
$\mathcal U(\underline a)  = -\sum_j \log \mathcal Z(a_j) $, and 
$\mathcal Z(a)$ is the energy density distribution under $m_0$, \
that behaves like $a$ for small $a$. 

Here is our main result:
 
\begin{thmm}\label{thm:main}
In the limit $\ve\to 0$, the law of $\{\mcEe_i\}_{i\in\Lambda}$ converges to the weak solution of the stochastic differential equations
\begin{equation}\label{eq:opla}
d\mcE_i=\sum_{k:|i-k|=1}\alpha(\mcE_i,\mcE_k) dt+
\sum_{k:|i-k|=1}\gamma(\mcE_i,\mcE_k) dB_{\{i,k\}} 
\end{equation}
with $B_{\{i,k\}} = - B_{\{k,i\}}$ independent standard Brownian
motions. Where the law of $\mcE_i(0)$ is given by the marginal of $F_\ve dm_\ve$ on the $\cE^\ve$.
\end{thmm}

Notice that the generator of the diffusion \eqref{eq:opla} on $\mathbb R_+^{\Lambda}$ is given by
\begin{equation}
  \label{eq:genene}
  \mathcal L = \sum_{|k-i|=1} \left(\gamma(\mcE_i,\mcE_k)^2 (\partial_{\mcE_i}- \partial_{\mcE_k} )^2 +
    \alpha(\mcE_i, \mcE_k) (\partial_{\mcE_i}- \partial_{\mcE_k} )
  \right) .
\end{equation}
Since $\gamma^2$ and $\alpha$ are smooth function on $\mathbb R_+^2$,
the uniqueness of the weak solution of \eqref{eq:opla} follows by applying the results in \cite{cerrai}.

Observe that there is a family of product probability measures
\begin{equation}
  \label{eq:flat}
  \prod_{i\in\Lambda} \mathcal Z(a_i) e^{-\beta a_i} N(\beta)^{-1} da_i =
  \prod_{i\in\Lambda} e^{- (\beta a_i + \mathcal U(a_i))} 
  N(\beta)^{-1} da_i, \qquad \beta >0 
\end{equation}
that are stationary and reversible for the diffusion generated by (\ref{eq:genene}).

As we will see shortly the proof of Theorem \ref{thm:main} relays
heavily on the fact that the unperturbed microscopic dynamics
$P_{\Lambda}^t$ generated by $L_{0,\Lambda}$ has strong mixing
properties ({\em hypocoercivity}). This is itself a non trivial result which we believe worth stating separately. Let $\cH_{\underline a}^1$ be the Sobolev space of order one on $\Sigma_{\underline a}$ with respect to the microcanonical measure $\mu_{\underline a}$ and a properly rescaled Riemannian structure (see Section \ref{sec:hypo} for more details). 

\begin{thmm}\label{thm:gap}
 For each set $\Lambda\subset \bZ^d$, there exists $C,\tau>0$ such that, for each energies $\underline a\in(0,\infty)^{\Lambda}$ and $\sigma\in(0,1)$,  the following holds true
\begin{itemize}
\item The semigroup $P^t_{0,\Lambda}$ is contractive in $L^2(\Sigma_{\underline a},\mu_{\underline a})$.
\item  For each smooth function $f\in \cH^1_{\underline a}$ such that $\mu_{\underline a}(f)=0$, holds 
\[
\| P^t_{0,\Lambda}f\|_{\cH^1_{\underline a}}\leq C e^{-\tau \sigma^2 t}\|f\|_{\cH^1_{\underline a}}.
\]
\end{itemize}
\end{thmm}
The proof is given in section \ref{sec:hypo}.

Before discussing the proof of the above results let us indulge in several remarks.
\begin{rem}
By \eqref{eq:ag} we can rewrite the generator as
\begin{equation}\label{eq:gene}
   \mathcal L =  \sum_{|k-i|=1}
     e^{ \mathcal U(\underline a)} 
     (\partial_{\mcE_i}- \partial_{\mcE_k} )
    e^{- \mathcal U(\underline a)}
    \gamma^2(\mcE_i,\mcE_k) (\partial_{\mcE_i}- \partial_{\mcE_k} )
\end{equation}
\end{rem}

\begin{rem}
The process \eqref{eq:opla} is close the the one studied by Varadhan in \cite{Va}, yet it is not covered by such results (due to the degeneracy at zero of the diffusion coefficients and the non strict convexity of the potential of the invariant measure).
In any case, the extension of Varadhan's work to the present case
would allow to obtain the {\em heat equation} in the present setting
via a diffusive scaling limit of space and time (\emph{hydrodynamic limit}).
\end{rem}

\begin{rem}
Note that both $\gamma$ and $\alpha$ depend on $\sigma$. One can wonder if equation \eqref{eq:opla} does admit a limit for small noise. Indeed, if $U$ and $V$ are quadratic then both $\gamma^2$ and $\alpha$ are proportional to $\sigma^{-2}$, see \eqref{eq:2bis}.  Hence the energy exchange for small noise is faster than the time scale we are exploring. This is due to the fact that in the quadratic case the solutions are quasiperiodic. On the other hand, for each $U, V$ a positive measure of such quasiperiodic solutions will survive by KAM theorem at least for small energies so it may be possible that a small noise limit exists, upon rescaling time, as in the quadratic case. The present results allow only upper bounds (which agree with the quadratic case), but it is unclear if a sufficiently exact scaling still exists.
\end{rem}

\section{Proof of Main theorem}\label{sec:proof}
This section is devoted to proving Theorem \ref{thm:main} by using several results detailed in the later sections (more precisely we assume Proposition \ref{lem:gap} and Lemmata \ref{lem:diff-global},  \ref{lem:l2loc}, \ref{lem:alphagamma} and \ref{lem:regular-ag}).
Our strategy is  the first establish tightness and then to show that the accumulation points satisfy \eqref{eq:opla}. Since \eqref{eq:opla} has a unique solution, the process have a unique accumulation point whereby proving the existence of the limit.

\subsection{Tightness}
Here we start studying the processes $\{\mcEe_i(t)\}$ defined in \eqref{eq:process}. 
\begin{lem}\label{lem:tight} There exists $\ve_0>0$ such that for each $T>0$, the processes $\{\mcEe_i(t)\}_{t\leq T}$, $0\leq \ve\leq \ve_0$, are tight.
\end{lem}
\begin{proof}
The proof use a \emph{standard} backward/forward martingale
decomposition argument (cf. \cite{svy,klo}). We recall it here.

Let us start the process with the equilibrium distribution $m_\ve$. 
Then the time reversed process, in a given time interval, is a Markov
process with generator given by  the adjoint $L_\ve^* = -A_\ve +
\sigma^2 S$. 
Remark that, since $X_k^2 p_k = - p_k$, we have $Sj_{i,k} =- j_{i,k}$.
So we can decompose
\begin{equation*}
  \begin{split}
     \ve \int_0^{t\ve^{-2}} j_{i,k}(s) \; ds = 
       - \frac \ve{2\sigma^2} \int_0^{t\ve^{-2}} L_\ve j_{i,k}(s) \;
       ds  - \frac \ve{2\sigma^2} \int_0^{t\ve^{-2}} L_\ve^*
       j_{i,k}(s) \; ds  \\
       = \frac{\ve}{2\sigma^2} M^+_{t\ve^{-2}} +
       \frac{\ve}{2\sigma^2} M^-_{t\ve^{-2}} 
  \end{split}
\end{equation*}
where $M^{\pm}_t$ are continuous martingales, adapted respectively to
the forward and backward filtration, that can be represented by the
stochastic integrals
\begin{equation*}
  M^{\pm}_t = \int_0^t (X_i j_{j,k})(s) dw^{\pm}_i(s) +  
  \int_0^t (X_k j_{j,k})(s) dw^{\pm}_k(s)  
\end{equation*}
where $w^+_i(t)$ and $w^-_i(t)$ are standard Wiener processes adapted
respectively to the forward and the backward filtration. 
Consequently the tightness follows from the tightness of each of these
stochastic integrals. Noticing that 
$X_k j_{j,k} = Jp_k\cdot \nabla V(q_j - q_k)/2$ is in $L^p(m_\ve)$ for
any $p<\infty$, by  
Doob's inequality:
\begin{equation*}
  \mathbb E_{m_\ve} \left( \sup_{0\le t\le T} (M^\pm_t)^4\right) \le
  \left(\frac 43\right)^4  \mathbb E_{m_\ve} \left(
    (M^\pm_T)^2\right)^2 \le C T^2 
\end{equation*}
This imply that the Kolmogorov criterion for tightness is
satisfied:
\begin{equation*}
   \sup_\ve \mathbb E_{m_\ve} \left(  | \mcEe_i(s) - \mcEe_i(t)|^4
   \right) \le  
   \frac{C'}{\sigma^4} (t-s)^2
\end{equation*}

In non-equilibrium, because of the assumption that the initial distribution $F$ is in
$L^2(m_\ve)$, the above argument extends immediately by a simple use of Schwarz inequality. 
\end{proof}

Observe that with the same argument we can also establish the
following bound for any $p>1$:
\begin{equation}\label{eq:backmart}
  \begin{split}
    \mathbb E_{m_\ve} \left( \sup_{0\le t\le T}\left[ 
        \int_0^t j_{i,k}(s) \; ds \right]^p \right) \le 
\frac {C T^{p/2}}{\sigma^p}
  \end{split}
\end{equation}
where $C$ is a constant independent of $\ve, T, \sigma$. 

Once we have the tightness all is left to prove is that the limit is unique, this is the content of the next section.

\subsection{Identification of the limit}\label{sec:limit}
The goal of this section is to prove that any accumulation point of
the laws of  $\{\mcEe_i(t)\}_{i\in\Lambda}$ must satisfy equation
\eqref{eq:opla}. More precisely, 
since 
\begin{equation}
\label{time-ev}
    \mcEe_i(t)- \mcEe_i(0) =\sum_{k:|k-i|=1}\ve
    \int_0^{\ve^{-2}t} j_{i,k}(s) \; ds 
\end{equation}
we want to show that, for every couple $i,k\in \Lambda$ such that
$|i-k|= 1$ 
 there exist orthogonal martingales
$\cM_{i,k}^\ve=-\cM_{k,i}^\ve$ with zero average and quadratic variance given by
\begin{equation}\label{eq-varqua}
  2 \int_0^t \gamma^2(\mcEe_i(s),\mcEe_k(s))\; ds
\end{equation}
and such that, for each $t\geq 0$, 
\begin{equation} \label{eq:33}
   \lim_{\ve \to 0} \mathbb
  E_{F}^\ve\left(\left|  \ve \int_0^{\ve^{-2}t} j_{i,k}(s) \; ds 
      - \int_0^{t} \alpha(\mcEe_i(s),\mcEe_k(s))
    ds +  \mathcal M_{i,k}^\ve(t)\right|\right) =0. 
\end{equation}

To prove \eqref{eq:33} it turns out to be useful to introduce a cutoff for high energy
$\xi_K = \chi_K(\mathcal E^0)$, where 
$\chi_K:\mathbb R_+^\Lambda \to[0,1]$ 
are smooth positive functions with support in
$[0,K+1]^\Lambda$, such that $\chi_K(a) = 1$ if $a\in
[0,K]^{\Lambda}$.  Then define $j^K_{i,k} =  j_{i,k}\xi_K$.

By the symmetry between $p$ and $-p$ it follows that
$\mu_{\underline a}(j_{i,k}^K)=0$ for each 
$\underline a\in \mathbb R_+^\Lambda$ and any
$K>0$. 
Arguing as in estimate (\ref{eq:backmart}), we have that
\begin{equation}  
\label{eq:cutoff}
  \lim_{K\to\infty}  \sup_{\ve>0} 
  \mathbb E_{m_\ve} \left( \sup_{0\le t\le T}\left[ 
       \ve \int_0^{t\ve^{-2}} j_{i,k}(1-\xi_K)(s) \; ds \right]^2
   \right)    = 0
\end{equation}

By Proposition \ref{lem:gap}, it follows that the equation
\begin{equation}
\label{poisson}
L_0 u_{i,k} = j_{i,k}
\end{equation}
 has a unique solution with zero average with respect to all measures $\mu_{\underline a}$. Note that $j_{i,k}=-j_{k,i}$, thus also $u_{i,k}=-u_{k,i}$.
 
In addition, Lemma \ref{lem:diff-global} implies $u_{i,k}\in\cC^\infty([\bR^4\setminus\{0\}]^2,\bR)\cap \cC^0(\bR^8,\bR)$. Observe that, since $L_0$ conserves 
 the energies of all particles, denoting $u_{i,j}^K = u_{i,j} \xi_K$
 we have
 \begin{equation}
\label{poissoncut}
L_0 u_{i,k}^K = j^K_{i,k}.
\end{equation}

We can thus write $L_\ve=L_0+\ve L_*$ with
\begin{equation}
L_*f = \frac 12 \sum_{|i-j|=1}\nabla V(q_i-q_j) \cdot 
(\partial_{p_i}-\partial_{p_j})f 
= \sum_{|i-j|=1} \nabla V(q_i-q_j)\cdot\partial_{p_i}f,\label{eq:Lstar}
\end{equation}

So we have
\begin{equation}\label{eq:start_dec}
L_\ve u_{i,k}^K = j_{i,k}^K + \ve L_* u_{i,k}^K, 
\end{equation}

\bigskip

Hence, denoting a path by $\omega_s = (q(s),p(s))$, we have
\begin{equation}\label{eq:int}
  \begin{split}
&   \ve \int_0^{\ve^{-2}t} j_{i,k}^K(s) \; ds
=  \ve \int_0^{\ve^{-2}t} 
    \left[ L_\ve u^K_{i,k} - \ve L_*u^K_{i,k} \right] (\omega_s) 
    ds
 \\
   &=  \bigg(\ve u^K_{i,k}(\omega_{\ve^{-2}t}) 
   - \ve u^K_{i,k}(\omega_0)
    - M^{\ve,K}_{i,k} (t)
    - \ve^2\int_0^{\ve^{-2}t}
   L_*u^K_{i,k}(\omega_s)\; ds\bigg),
  \end{split}
\end{equation}
where we have introduced the centered martingale
$$
M^{\ve,K}_{i,k} (t):= \ve u^K_{i,k}(\omega_{\ve^{-2}t})- \ve
u^K_{i,k}(\omega_0) - \ve\int_0^{\ve^{-2}t}  
L_\ve u^K_{i,k}(\omega_s) ds.
$$
Due to the property $u_{i,k}=-u_{k,i}$ we have $M^{\ve,K}_{i,k}=-M^{\ve,K}_{k,i}$ and the same for all the derived martingales.

The quadratic variations of  $M^{\ve,K}_{i,k}$ is given by
\begin{equation}\label{eq:var}
  \begin{split}
    \langle\langle M^{\ve,K}_{i,k}, M^{\ve,K}_{i',k'} \rangle\rangle (t) = 2\sigma^2
    \sum_{j\in\Lambda} 
    \ve^2\int_0^{\ve^{-2}t} (X_{j} u^K_{i,k} ) (X_{j} u^K_{i',k'}
    )(\omega_s) ds \\
     = 2\sigma^2
    \sum_{j\in\{i,k\}\cap\{i',k'\}}
    \ve^2\int_0^{\ve^{-2}t} \xi_K^2 (X_{j} u_{i,k} ) (X_{j} u_{i',k'}
    )(\omega_s) ds.
  \end{split}
\end{equation}

In order to close the evolution equations, as $\ve \to 0$, 
we need to prove that
\begin{equation}
  \label{eq:4}
  \lim_{\ve\to 0} \mathbb E^\ve_{m_\ve} \left|
    \ve^2\int_0^{\ve^{-2}t}\big[L_*u^K_{i,k}(\omega_s) -\mu_{
      \underline{\mc E}^0(\omega_s)}(L_*u^K_{i,k})\big] \; ds \right|^2 = 0
\end{equation}
and that 
\begin{equation}
  \label{eq:5}
   \lim_{\ve\to 0} \mathbb E^\ve_{m_\ve} \left|
     \ve^2\int_0^{\ve^{-2}t}\xi_K^2 \left[ (X_{j} u_{i,k} ) (X_{j} 
       u_{i',k'})(\omega_s)
     -\mu_{\underline{\mc E}^0(\omega_s)}((X_{j} u_{i,k} ) (X_{j} 
     u_{i',k'}))\right]
    \; ds \right|^2 = 0 
\end{equation}
These are consequence of Lemma \ref{lem:l2loc}, \ref{lem:diff-global} and the following lemma:

\begin{lem}\label{lem:average}
Let $f(\omega)$ a function in $L^2(m_0)$ 
such that $\mu_{\underline a}(f)=0$ for all $\underline a\in\bR_+^{|\Lambda|}$,
then
\[
\lim_{\ve\to 0}\mathbb E_{m_\ve}^\ve\left|\ve^{2}\int_0^{\ve^{-2}t}
  f(\omega_s)ds\right|^2  =0 .
\]
\end{lem}

\begin{proof}

  By using Jensen's inequality and stationarity, it is enough to prove that 
  \begin{equation}
    \label{eq:15}
    \lim_{T\to \infty}  \lim_{\ve\to 0}
    \mathbb E_{m_\ve}^\ve\left(\left|\frac 1T \int_0^{T} 
      f(\omega_s)ds\right|^2\right) =0
  \end{equation}
  Since $\lim_{\ve\to 0}
  \mathbb E_{m_\ve}^\ve =  \mathbb E_{m_0}^0$, i.e. the expectation with
  respect to the dynamics of the non-interacting oscillators starting 
  with the product of canonical measures $m_0$, which are convex  
  combination of the microcanonical ones. Then the result follows by
  the ergodicity of the dynamics of the single oscillators.  
\end{proof}

So far we have obtained that 
\begin{equation}\label{kpartiale}
  \lim_{\ve \to 0} \mathbb
  E_{m_\ve}^\ve\left(\left|\ve^{2}\int_0^{\ve^{-2}t}  
       \left(\ve^{-1} j_{i,k}^K(s) - \mu_{
      \underline{\mc E}^0(\omega_s)}(L_*u^K_{i,k})\right)
    ds -  \hat M_{i,k}^{\ve,K}(t)\right|^2\right) =0 ,
\end{equation}
where the martingales $\hat M_{i,k}^{\ve,K}$ have  quadratic variation given by
\begin{equation}\label{eq:var-2}
  \begin{split}
    \langle\langle \hat M^{\ve,K}_{i,k}, \hat M^{\ve,K}_{i',k'} \rangle\rangle (t) 
     = 2\sigma^2\!\!\!\! \!\! \sum_{j\in\{i,k\}\cap\{i',k'\}}
     \xi_K({\underline\cE}_0(s))^2 \mu_{{\underline\cE}_0(s)}\left((X_{j} u_{i,k} ) (X_{j} u_{i',k'})\right).
  \end{split}
\end{equation}
Next, we remove the cutoff on the energies. Observe that 
\begin{equation*}
  L_* u_{i,k}^K = \xi_K L_* u_{i,k} + u_{i,k}  L_* \xi_K\,,
\end{equation*}
which, writing $\xi_K^{(j)}(\underline a)$ for $\partial_{a_j}\xi_K(\underline a)$, implies
\begin{equation*}
  \begin{split}
     \mu_{\underline a}(L_*u^K_{i,k}) = 
     \xi_K(\underline a)  \mu_{\underline a}(L_*u_{i,k}) +
     2\sum_{|i-j|=1} \xi_K^{(j)}(\underline a)  
     \mu_{\underline a}(u_{i,k}\nabla V(q_i - q_j)\cdot p_j)\\
     =  \xi_K(\underline a)  \mu_{\underline a}(L_*u_{i,k}) +
     2 \xi_K^{(k)}(\underline a)  
     \mu_{\underline a}(u_{i,k}\nabla V(q_i - q_k)\cdot p_k)\\
      =  \xi_K(\underline a)  \mu_{\underline a}(L_*u_{i,k}) +
      \xi_K^{(k)}(\underline a)  \gamma^2(a_i,a_k)
 \end{split}
\end{equation*}
Since $m_\ve(\mc E_j \ge K)$ is exponentially small in $K$, and
$\gamma^2$ does not grow faster than polynomially, we have
\begin{equation*}
  \begin{split}
   \sup_\ve \mathbb E_{m_\ve}^\ve\left(\left|\ve^{2}\int_0^{\ve^{-2}t}
        \left(\mu_{\underline{\mc E}^0(\omega_s)}(L_*u^K_{i,k})-
          \xi_K(\underline {\mc E}^0(\omega_s))
          \mu_{\underline{\mc E}^0(\omega_s)}(L_*u_{i,k})\right)
         ds\right|^2\right)\\  
    \le C\sup_\ve \sum_j m_\ve \left( \gamma^2({\mc E}^0_i,{\mc
        E}^0_k)   \Id_{[\mc E_j \in
        (K,K+1)]}\right)K^2\  
    \mathop{\longrightarrow}_{K\to \infty} \ 0 .
  \end{split}
\end{equation*}

Since the above limit and \eqref{eq:cutoff} are uniform in $\ve$, and
since $X_j u_{i,k}$ is in $L^2(m_\ve)$, we can take the limit $K\to\infty$ in  (\ref{kpartiale}) and obtain
\begin{equation}\label{palle12}
  \lim_{\ve \to 0} \mathbb
  E_{m_\ve}^\ve\left(\left|  \ve \int_0^{\ve^{-2}t} j_{i,k}(s) ds  
      -  \ve^{2}\int_0^{\ve^{-2}t}  
         \mu_{\underline{\mc E}^0(\omega_s)}(L_*u_{i,k})
    ds +  M_{i,k}^{\ve}(t)\right|^2\right) =0, 
\end{equation}
where $M_{i,k}^{\ve}(t)$ is a martingale whose quadratic variation has
the expression given by (\ref{eq:var-2}) with $\xi_K$ substituted by
$1$. 

\begin{lem} \label{lem:zero}
For each $a_k, q_i,p_i$ setting
\[
\bar u_i(q_i,p_i):=\int u_{k,i}(q_k,p_k,q_i,p_i) \mu_{a_k}(d q_k,dp_k)
\]
holds true
\[
\bar u_i(q_i,p_i)=\int V(q_i-q_k)\mu_{a_k}(d q_k,dp_k)-\int V(q_i'-q_k)\mu_{a_k}(d q_k,dp_k)\mu_{\cE^0_i}(d q_i',dp_i').
\]
In particular, $\bar u_i(q_i,-p_i)=\bar u_i(q_i,p_i)$ and $X_i\bar u_i=0$.
\end{lem}
\begin{proof}
First of all note that, due to the symmetry between $p$ and $-p$,
\[
\int j_{k,i}(q_k,p_k,q_i,p_i) \mu_{a_k}(d q_k,dp_k)=\int \partial_{q_i} V(q_i-q_k)p_i\mu_{a_k}(d q_k,dp_k).
\]
Thus, setting $\bar V(q_i)=\int  V(q_i-q_k)\mu_{a_k}(d q_k,dp_k)$,
\[
\int j_{k,i}(q_k,p_k,q_i,p_i) \mu_{a_k}(d q_k,dp_k)=L_{0,\{i\}}\bar V(q_i, p_i).
\]
In addition, by the product structure of the generator,
\[
j_{k,i}=L_{0,\Lambda} u_{k,i}=L_{0,\{i\}}u_{k,i}+L_{0,\{k\}}u_{k,i}.
\]
Integrating the above we have
\[
L_{0,\{i\}}\bar V(q_i, p_i)=\int \left[ L_{0,\{i\}}u_{k,i}+L_{0,\{k\}}u_{k,i}\right]\mu_{a_k}(d q_k,dp_k).
\]
Since $\mu_{a_k}$ is the invariant measure of $L_{0,\{k\}}$, we have
\[
L_{0,\{i\}}\bar V(q_i, p_i)= L_{0,\{i\}}\bar u_{i}.
\]
By Proposition \ref{lem:gap}, applied with $\Lambda=\{i\}$ it follows that the only solutions of the above equation are of the form $\bar u_{i}=\bar V+f(\cE^0_i)$ for some function $f$. Next, since $\bar u_i$ is of zero average by construction, $f(\cE^0_i)=-\mu_{\cE^0_i}(\bar V)$.
\end{proof}
Thus, if $l\not\in\{i,k\}$ we have 
\[
\begin{split}
\iiint 
\partial_{p_i}u_{i,k} \cdot \nabla V(q_i-q_l) \; 
\mu_{a_k}(dq_k,dp_k)\; \mu_{a_i}(dq_i,dp_i)\; \mu_{a_l}(dq_l,dp_l) \\
= \iint  \partial_{p_i}\bar u_{i} \cdot \nabla V(q_i-q_l)\; \mu_{a_i}(dq_i,dp_i)\; \mu_{a_l}(dq_l,dp_l)=0 
\end{split}
\]
due to the antisymmetry of $\partial_{p_i}\bar u_{i}$ with respect to $p_i$ established in Lemma \ref{lem:zero}.
From this follows
\begin{equation}\label{eq:sarp-ave}
\mu_{\underline a}(L_* u_{i,k}) 
=\mu_{\underline a}\left[\nabla V 
(q_i-q_k)(\partial_{p_i}-\partial_{p_k})u_{i,k}\right]
  = :\alpha(a_i,a_k). 
\end{equation}
In fact, again by the product structure, the above is a function of
$a_i,a_k$ only. 

It is also convenient to define
\begin{equation}\label{eq:gamma}
  \gamma^2(a_i, a_k)= \sigma^2\mu_{\underline a} \left((X_{i} u_{i,k} )^2+(X_{k} u_{i,k} )^2\right)=-\mu_{\underline a} \left(u_{i,k}j_{i,k}\right).
\end{equation}
Accordingly, by (\ref{eq:var-2}) and Lemma \ref{lem:zero}, 
\begin{equation}\label{eq:var-3}
    \langle\langle \hat M^{\ve}_{i,k}, \hat M^{\ve}_{i',k'} \rangle\rangle (t) 
     = 2\sigma^2   \sum_{j\in\{i,k\}}
    \mu_{{\underline\cE}_0(s)}\left((X_{j} u_{i,k} )^2\right)
    \left(\delta_{i=i', k= k'}-\delta_{i=k', k= i'}\right)
\end{equation}

\begin{lem}\label{lem:pert}
  \begin{equation*}
   \lim_{\ve \to 0} \mathbb E_{m_\ve}
   \left| \int_0^t \left(\alpha(\mc E^0_i(\omega_{\ve^{-2}s}),
      \mc E^0_k(\omega_{\ve^{-2}s})) -   \alpha(\mcEe_i(s),
      \mcEe_k(s)) \right)\; ds \right|^2 = 0
  \end{equation*}
and similarly for $\gamma^2$.
\end{lem}

\begin{proof}
  By stationarity and Schwarz inequality we have
\[
\begin{split}
\mathbb E_{m_\ve} &\left| \int_0^t \left[\alpha(\mc E^0_i(\omega_{\ve^{-2}s}),
      \mc E^0_k(\omega_{\ve^{-2}s})) -   \alpha(\mcEe_i(s),
      \mcEe_k(s)) \right]\; ds \right|^2\\
      &\leq t m_\ve  \left(|\alpha(\mc E^0_i,
      \mc E^0_k) -   \alpha(\mc E^{\ve}_i,
      \mc E^{\ve}_k)|^2 \right)
\end{split}
\]
Since we prove in Lemmata \ref{lem:alphagamma}, \ref{lem:regular-ag} that $\alpha(a_1, a_2)$ is uniformly Lipschitz in $\bR_+^2$ the result follows by elementary arguments. 
\end{proof}

Applying Lemma \ref{lem:pert} to \eqref{palle12} and \eqref{eq:var-3} we obtain that
\begin{equation}\label{palle13}
  \lim_{\ve \to 0} \mathbb
  E_{m_\ve}^\ve\left(\left|  \ve \int_0^{\ve^{-2}t}\!\!\!\!\! j_{i,k}(s) ds 
      - \int_0^{t}\!\!\! \alpha(\mcEe_i(s),\mcEe_k(s))
    ds +  \mathcal M_{i,k}^\ve(t)\right|^2\right) =0, 
\end{equation}
which, remembering \eqref{eq:gamma}, yields the wanted result in equilibrium.
Our general claim \eqref{eq:33} follows by Schwarz inequality.

\section{Hypoellipticity and regularity on energy shells}
\label{sec:single-site-dynam}

We start here by studying the single site dynamics. 
Here $p= (p^1,p^2),\ q=(q^1,q^2)$ will be the coordinates. We will use
$L$ to designate the generator \eqref{eq:generator} for $\ve=0$ and
$\Lambda$ consisting of only one site. Since $L$ preserves the energy,
we can look at its action on each energy surface 
$$
\Sigma_a = \{ (q,p)\in\bR^4 \;:\; \frac{|p|^2}2 +  U(q) =a \}, \qquad
a>0.
$$
\begin{lem}\label{lem:hypoell}
For each $a>0$ the operator $L$ is hypoelliptic on $\Sigma_a$.
\end{lem} 
\begin{proof}
We must study of the Lie algebra generated by the vector fields
\begin{equation}
  \label{eq:16}
  \begin{split}
    C_0 &= X = Jp \cdot \partial_p,\\
    B & = A_0 = p\cdot \partial_q - \nabla U(q) \cdot \partial_p .
  \end{split}
\end{equation}
We obtain
\begin{equation}
\begin{split}
  C_1 = : [C_0, B] =  Jp\cdot \partial_q + J\nabla U(q) \cdot \partial_p .
\end{split}\label{eq:17}
\end{equation}
\begin{equation}
  \label{eq:19}
  [C_1, B] = 2J \nabla U(q) \cdot \partial_q - \{D^2 U (q) J + J
  D^2U(q)\} p\cdot \partial_p  .
\end{equation}
By our assumption on $U$
\begin{equation}
  \label{eq:20}
  D^2 U(q) = 4 \oU''(q) \,q\otimes q +  2\oU'(q) \Id .
\end{equation}
Since
\begin{equation*}
J (q\otimes q) +  (q\otimes q) J = |q|^2J ,
\end{equation*}
we have
\begin{equation*}
  D^2 U (q) J + J D^2U(q) = 4\{\oU''(q)|q|^2+\oU'\} J=:\zeta(q)J .
\end{equation*}
Finally we define $C_2$ by the relation
\begin{equation*}
   [C_1, B] = 4 \oU'(q)Jq \cdot \partial_q - 4\{\oU''(q)|q|^2+\oU'\}  Jp \cdot \partial_p
   = : 2 C_2 -\zeta(q)  C_0  .
\end{equation*}
Observe that, setting $\cN= (|p|^2 + |\nabla U(q)|^2)^{1/2}$, the vectors
\begin{equation}\label{eq:ort-base}
  Z_1 = \frac {C_1}{\cN} , \qquad  Z_0 = \frac {B}{\cN},  
  \qquad Z_2 = \frac {C_2 - C_0}{\cN}
\end{equation}
form an orthonormal base of the tangent space of $\Sigma_{a}$, hence the Lie Algera generated by $\{C_0, B\}$ spans the space of any energy shell $\Sigma_a$. 
This concludes the proof of the hypoellipticity of $L$.
\end{proof}
By the above results trivially follows the claimed hypoellipticity.
\begin{lem}\label{lem:hypo-many}
For each $\Lambda\subset \bZ^d$,  $a_i>0$, $i\in\Lambda$ the operator $L_{0,\Lambda}$ is hypoelliptic on $\Sigma_{\underline a}$.
\end{lem}
By H\"ormander theorem \cite{horm} hypoellipticity implies that if
there exists a solution  $u$ for the equation  
  \begin{equation}\label{eq:pois-g}
  L_{0,\Lambda} u = g
\end{equation}
where $g$ is a $\cC^\infty$ function when restricted to any energy shell, then also $u$ is $\cC^\infty$ when restricted to any energy shell.

\section{Hypocoercivity}\label{sec:hypo}

We will prove the existence of the solution of equation  \eqref{poisson} by proving a spectral gap for the generator $L_0$ on each energy shell $\Sigma_{\underline a}$ in a proper Hilbert space.
More precisely we consider the Hilbert spaces determined by the scalar products
\begin{equation}\label{eq:hilbert}
  \begin{split}
    \langle h,g\rangle_{\scH^1_{\underline a}}&:=|\Lambda|\langle
    h,g\rangle_{\underline a}+\sum_{l\in\Lambda}  \mathcal
    D_{\underline a,l} (h,g) 
    \\
    \mathcal D_{\underline a,l} (h,g) &:= \sum_{k=0}^2\langle
    C_{k,l}h,C_{k,l} g\rangle_{\underline a}+\langle B_lh,
    B_lg\rangle_{\underline a}.
  \end{split}
\end{equation}
where $\langle h, g\rangle_{\underline a}=\int_{\Sigma_{\underline a}}\overline{h} gd\mu_{\underline a}$ and the $C_{k,l}, B_l$ are the vector fields defined in section \ref{sec:single-site-dynam} relative to the particle $(q_l,p_l)$.

By a slight abuse of notations we will use $P^t_0$ to denote the strongly continuous semigroup generated by $L_0$ both in $L^2$ and $\scH^1_{\underline a}$ for each ${\underline a}$.

Note that the above norm is equivalent to the standard Sobolev space $\cH_{\underline a}$ on the Riemannian surface $\Sigma_{\underline a}$ where the Riemannian structure has been rescaled to have the diameter of each $\Sigma_{a_i}$ equal one independently of the $a_i$. More precisely, there exists $c>0$ such that
\begin{equation}\label{eq:norm-equiv}
c\|f\|_{\cH_{\underline a}}\leq \|f\|_{\scH_{\underline a}}\leq c^{-1}|\Lambda|\,\|f\|_{\cH_{\underline a}}
\end{equation}
Accordingly, Theorem \ref{thm:gap} is a direct consequence of the following.
 
\begin{prop}\label{lem:gap}
The semigroup $P^t_0$ is contractive in $L^2$.
In addition, for each $\Lambda\subset \bR^d$, $|\Lambda|<\infty$, there exists $C,\tau>0$ such that for all $\underline{a}\in(0,\infty)^\Lambda$ and smooth function $f$, such that $\mu_{\underline a}(f)=0$ , holds 
\[
\| P^t_{0}f\|_{\scH^1_{\underline a} }\leq Ce^{-\tau t}\|f\|_{\scH^1_{\underline a}}.
\]
\end{prop}
\begin{proof}
The contractivity in $L^2$ follows by
\[
\frac d{dt}\|P_0^t f\|_{\underline a}^2\leq -2\sum_L\langle C_{0, l}P_0^tf,C_{0, l}P_0^tf\rangle_{\underline a}  \leq 0.
\]

To prove the second part of the theorem we use the results of Appendix \ref{sec:hypoest}. 
Note that for all smooth $h,g,\vf$ we can write
\[
m_0(\vf(\underline \cE)\mu_{\underline \cE}(h C_0g))=m_0(\vf(\underline \cE)h C_0g)=-m_0(\vf(\underline \cE)g C_0h),
\]
hence it must be $\mu_{\underline a}(h C_0g)=-\mu_{\underline a}(C_0h
\cdot g) $, and the same for $B$. We can thus use  Lemma
\ref{lem:hypo} applied to, for each $l\in\Lambda$, a scalar product
$\lv\cdot,\cdot\rv_{{\underline a},l}$ defined in \eqref{eq:twoangles} using the operators $C_{k,l},B_l$. By Lemma \ref{lem:positive} such a scalar product is equivalent to the one 
\[
\langle h,g\rangle_{\scH^1_{{\underline a},l}}:=\langle h,g\rangle_{\underline a}+ \mathcal D_{\underline a,l} (h,g).
\]

Then Lemma \ref{lem:hypo} states, that
\begin{equation}\label{eq:gap-l}
\lv h, L_{\{l\}} h\rv_{{\underline a},l}\leq -\tau\sigma^2 \left(
  \mathcal  D_{\underline a,l} (h,h) +  \mathcal D_{\underline a,l}
  (C_0h,C_0h)\right) \le  -\tau\sigma^2 
  \mathcal D_{\underline a,l} (h,h) .  
\end{equation}
where $L_{\{l\}}$ is the generator of the dynamics of the isolated
atom $l$, i.e. $L_0 = \sum_{l\in\Lambda} L_{\{l\}}$. 

The last piece of information we need is given by the following
Poincare inequality. 

\begin{lem} \label{lem:poinc} There exist a constant $K_P>0$ such that for each $\underline a$, if $f\in\scH^1_{\underline a}$, such that $\mu_{\underline a}(f)=0$, then 
\[
\|f\|_{\underline a}^2 \leq K_P \sum_{l\in\Lambda} \mathcal
D_{\underline a,l} (f,f).  
\]
\end{lem}
\begin{proof}
By the change of variable introduced in section~\ref{sec:coordinate}
one can use the Poincar\'e inequality for the sphere, for one
particle. After that 
\[
\begin{split}
  \|f\|_{\underline a}^2 &\leq \|f-\mu_{a_l}(f)\|_{\underline a}^2+
  \|\mu_{a_l}(f)\|_{\underline a}^2 \\
  & \leq K_P D_{\underline a,l} (f,f)
  +\|\mu_{a_l}(f)-\mu_{a_l,a_j}(f)\|_{\underline a}^2+ \|
  \mu_{a_l,a_j}(f)\|_{\underline a}^2
\end{split}
 \]
iterating the argument yields the result.
\end{proof}

It follows from (\ref{eq:gap-l}) and Lemma \ref{lem:poinc} that
\begin{equation}
  \label{eq:73}
  \begin{split}
    \lv h,L_0 h\rv &= \sum_{l,l'} \lv h,L_{\{l'\}} h\rv_{\underline
      a,l} \le \sum_{l} \lv h,L_{\{l\}} h\rv_{\underline a,l}-
       \sum_{l\neq l'} \lv C_{0,l'}h,C_{0,l'} h\rv_{\underline a,l} \\
       &\le \sum_{l} \lv h,L_{\{l\}} h\rv_{\underline a,l}\le 
       -\tau\sigma^2
       \sum_{l}  \cD_{\underline a,l} (h,h) \\ 
    &\le -\frac{\tau\sigma^2}{1+ K_P|\Lambda|} \|h\|^2_{\scH^1_{\underline a}}
  \end{split}
\end{equation}

Accordingly, for each $\underline a$ and $h\in\scH^1_{\underline a}$, such that $\mu_{\underline a}(h)=0$, we have
\[
\begin{split}
\frac d{dt}\lv P_0^th,P_0^th\rv_{\underline a}&=2\lv P_0^th,L_0P_0^th\rv_{\underline a}\\
&\leq -\frac{2\tau\sigma^2}{1+K_P|\Lambda|} \|P^t_0h\|^2_{\scH^1_{\underline a}}\\
&\leq - 2\tau_1\sigma^2 \lv P_0^th,P_0^th\rv_{\underline a},
\end{split}
\]
where, in the last line, we have used first \eqref{eq:gap-l} and then Lemma \ref{lem:positive}.
This means that
\[
\lv P_0^th,P_0^th\rv_{\underline a}\leq \lv h,h\rv_{\underline a}
e^{-2\tau_1\sigma^2 t} 
\]
and, by the equivalence of the norms, there exists $C, \tau'>0$ such
that 
\begin{equation}\label{eq:coercive}
\|P_0^th\|_{\scH^1_{\underline a}}\leq Ce^{-\tau'\sigma^2 t}\|h\|_{\scH^1_{\underline a}}.
\end{equation}

\end{proof}

\section{Global regularity}\label{sec:regular}
Next, we need regularity of the solution of \eqref{eq:pois-g} also in directions not tangent to the energy surfaces, provided $g$ is smooth. Again we work first with only one particle. 

\subsection{The transversal direction}
A natural direction would be given by the normal vector to the energy surfaces
\begin{equation}
  \label{eq:22}
  Z_3 = \frac{ p\cdot \partial_p + \nabla U(q) \cdot \partial_q}{\mathcal N},
\end{equation}
but, due to the anharmonicity of the potential $U$, it turns out to be more useful to work with the vector field
\begin{equation}\label{eq:auto}
Y:= \frac 1{2\sqrt{\cE^0(q,p)}}\left\{\frac 12p\cdot\partial_{p} +
\frac{U(q)}{\nabla U(q)\cdot q}q\cdot\partial_{q} \right\}.
\end{equation}
Note that,  $Y$ is a smooth vector field away from zero. Observe that in the harmonic case 
($\oU(r) =\frac 12 r$), $Y$ is parallel to the normal vector field $Z_3$.

The reason why we consider $Y$ is the following.
A direct computation shows that $Y\cE^0=\frac 12 \sqrt{\cE^0(q,p)}$, hence the vector field is transversal to the energy surface. In addition,
\begin{equation}\label{eq:Ycom}
\begin{split}
&[Y,C_0]=0,\\
&[Y,L]\cE^0=-\frac 12 L\sqrt{\cE^0(q,p)}=0.
\end{split}
\end{equation}
That is, the commutators are vector fields tangent to the constant
energy surface. Note that since $U''(0) >0$, we can write
 $\frac{U(q)}{q\cdot \nabla U(q)}=\frac
12+\kappa(|q|^2)$, for a smooth function $\kappa$ such that 
 $\kappa(0)=0$. An explicit computation yields 
\begin{equation}\label{eq:long}
\begin{split}
[L,Y]=  \frac{4\left(\frac 12+\kappa\right) \oU'' q^2+2\kappa \oU'}{2\sqrt{\cE^0(q,p)}}q\cdot\partial_p + \frac 1{2\sqrt{\cE^0(q,p)}}p\cdot \left[ 2\kappa'q\otimes q+\kappa\Id  \right] \partial_q  
\end{split}
\end{equation}
\begin{rem} Note that the vector field $(\cE^0(q,p))^{-\frac 12}[L,Y]$ is smooth on all $\bR^{2\nu}$, in particular even at zero. This fact will play a crucial role in the following.
\end{rem}

\subsection{Transversal regularity}
The basic idea to prove regularity is to notice that if $Lu=g$,
then one expects that $LYu=[Y,L]u+Yg$. Unfortunately, we have only
$L^2$ bounds for the right hand side of the above equations, in
particular we do not know if it belongs to $\scH^1$. Hence, a priori,
we do not know if such an equation has a solution in $\scH^1$. To
overcome such a difficulty several direct strategies are possible. For
example one could try to prove a spectral gap in Sobolev spaces of
higher regularity or to prove that the semigroup maps $L^2$ functions
in $\scH^1$ functions. Unfortunately, such results (even if probably
true) are not so easy to prove, in particular the related algebra
becomes quickly very messy. Due to this state of affair we take a bit
more indirect route that, without proving explicit bounds, suffices to
prove the smoothness. To this end it is convenient to work in
coordinates in which all the energy surfaces can be naturally
identified.

It is then natural to transform equation \eqref{poisson} in the following coordinates.
Let $S^3:=\{x\in\bR^4\;:\;\|x\|=1\}$ and $M=\bR\times S^3\subset \bR\times \bR^2\times\bR^2$, $M_+=(0,\infty)\times S^3\subset M$. Let $\Psi: \bR^{4|\Lambda|}\setminus\{0\}\to M_+^{|\Lambda|}\subset \bR^{5|\Lambda|}$ be defined by
\begin{equation}\label{eq:regularize}
\Psi(q,p)_i=\left(\sqrt{\frac{p_i^2}2+U(q_i)}, 
\frac{q_i\sqrt{U(q_i)}}{|q_i|\sqrt{\frac{p_i^2}2+U(q_i)}},
\frac{p_i}{\sqrt{{p_i^2} + 2U(q_i)}}\right)=:(r_i,\xi_i,\eta_i)
\end{equation}
The needed properties of this change of variables are detailed in section \ref{sec:coordinate}.

The first key observation is that the problem is now regularized at zero energy. Indeed,
\[
\begin{split}
\tilde j_{i,k}(\underline r,\underline \xi,\underline \eta)&=j_{i,k}\circ \Psi^{-1}(\underline r,\underline \xi,\underline \eta)\\
&=\frac 1{\sqrt 2}\nabla V\left(r_i\theta(r_i^2\xi_i^2)\xi_i-r_k\theta(r_k^2\xi_k^2)\xi_k\right)\cdot(r_i\eta_i+r_k\eta_k).
\end{split}
\]
\begin{rem}\label{rem:simmetry}
Note that $\tilde j_{i,k}$ extends naturally to a smooth function on $M^{\Lambda}$. Indeed if $r_i<0$, then we can set $\tilde j_{i,k}(\underline r,\underline \xi,\underline \eta)=\tilde j_{i,k}(\underline r',\underline \xi',\underline \eta')$ where $r_j=r_j'$, $\xi_j=\xi_j'$, $\eta_j=\eta_j'$ for all $j\neq i$ and $r_i=-r_i'$, $\xi_i=-\xi_i'$, $\eta_i=-\eta_i'$, and the same for $k$.
\end{rem}

Note that for each function $\tilde f\in\cC^\infty(M^{|\Lambda|},\bR)$ holds 
$C_0 (\tilde f\circ \Psi)=(\tilde C_0 \tilde f)\circ \Psi$. It follows that the equation $L_0f=g$ on $(\bR^4\setminus\{0\})^{|\Lambda|}$ is transformed in the equation $\tilde L_0\tilde f=\tilde g$ on $\Psi((\bR^4\setminus\{0\})^{|\Lambda|})\subset M$ where $\tilde L_0=\sum_l\sigma^2\tilde C_{0,l}^2+\tilde B_l$ and $\tilde f=f\circ\Psi^{-1}$, $\tilde g=g\circ\Psi^{-1}$.

It is then natural to study the equation on $M$
\begin{equation}\label{eq:poisson-reg}
\tilde L_{0}\tilde u_{i,k}=\tilde j_{i,k}.
\end{equation}
By the previous discussion the solution of \eqref{poisson} in $(\bR^{4}\setminus\{0\})^{\Lambda}$ is given by $u_{i,k}=\tilde u_{i,k}|_{M_+^{\Lambda}}\circ \Psi$.

The problem of the transversal smoothness is then reduced to studying the smoothness of $\tilde u_{i,k}$ in $r$ (see Lemma \ref{lem:regularize}). 

\begin{lem}\label{lem:diff-global} For each $i,k$, the functions $\tilde u_{i,k}\in\cC^\infty(M^{\Lambda},\bR)$.
\end{lem}
\begin{proof}
We can consider $\tilde C_i, \tilde B^k_i$ as vector fields on $S^{3|\Lambda|}$.\footnote{See \eqref{eq:btilde} for the definition of $\tilde B^k_i$.} Accordingly, we can define for each $\underline r\in\bR^{\Lambda}$
\[
\tilde L_0(\underline r)=\sum_{i\in\Lambda}  \left\{\sigma^2\tilde C_{0,i}^2+ \frac{\sqrt 2}{\theta(r_i^2\xi_i^2)}\big [\tilde B^0_i +\Gamma(r_i^2\xi_i^2)\tilde B_i^1\big] \right\}.
\]
Then, setting $\tilde j_{i,k,\underline r}(\underline \xi,\underline\eta):=\tilde j_{i,k}(\underline r,\underline \xi,\underline\eta)$ and considering the equation on $S^{3|\Lambda|}$
\begin{equation}\label{eq:poisson-micro}
\tilde L_0(\underline r) \tilde u_{i,k,\underline r}=\tilde j_{i,k,\underline r}
\end{equation}
it follows $\tilde u_{i,k}(\underline r,\underline \xi,\underline\eta)=\tilde u_{i,k,\underline r}(\underline \xi,\underline\eta)$. By the previous section we know that, for each $\underline r$, $\tilde u_{i,k,\underline r}\in\cC^\infty(S^{3|\Lambda|},\bR)$, thus the differentiability boils down to show that the solution of \eqref{eq:poisson-micro} are differentiable with respect to the parameter $\underline r$.

Let us fix $\underline r$ and consider the equation\footnote{This is nothing else than the formal derivative of \eqref{eq:poisson-micro} with respect to $r_l$.}
\begin{equation}\label{eq:poisson-itera}
\tilde L_0(\underline r) \tilde v_{i,k, l,\underline r}=\tilde g_{i,k,l, \underline r}
\end{equation}
where $\tilde g_{i,k,\underline
  r}(\underline\xi,\underline\eta)=\left([\tilde Y_l, \tilde L_0]
  \tilde u_{i,k}\right)(\underline
r,\underline\xi,\underline\eta)+(\tilde Y_l\tilde j_{i,k})(\underline
r,\underline\xi,\underline\eta)$. Clearly $\tilde g_{i,k,l,\underline
  r}\in\cC^\infty(S^{3|\Lambda|},\bR)$ for each choice of $i,k,l,\underline r$,
thus $\tilde v_{i,k,l,\underline r}\in\cC^\infty(S^{3|\Lambda|},\bR)$. We claim
that $\partial_{r_l}\tilde u_{i,k}=\tilde v_{i,k,l,\underline r}$, let
us prove it. 

For each, small, $h\in\bR^{\Lambda}$ we can write
\[
\begin{split}
\tilde L_0(\underline r+h)\left[ \tilde u_{i,k,\underline r+h}-\tilde u_{i,k,\underline r}-\sum_l \tilde v_{i,k, l,\underline r}h_l\right] =&\tilde j_{i,k,\underline r+h}-\tilde j_{i,k,\underline r}-\sum_l\tilde g_{i,k,l, \underline r}h_l\\
&-\left[\tilde L_0(\underline r+h) -\tilde L_0(\underline r)\right] \tilde u_{i,k,\underline r}.
\end{split}
\]
An explicit computation shows that the $\scH^1$ norm of the right hand
side is bounded by a constant (depending in an unknown manner form
$\underline r$) times $o(\|h\|)$. The by Proposition \ref{lem:gap} the
differentiability of $u_{i,k,\underline r}$ follows. 

Note that in the above argument the only relevant property of $\tilde j_{i,k,\underline r}$ is the smoothness on $S^{3|\Lambda|}$ of itself and of its derivatives with respect to $\tilde Y_l$. Since a direct computation shows that the same properties are enjoyed by $\tilde g_{i,k,l, \underline r}$,\footnote{The smoothness on $S^{3|\Lambda|}$ follows from the smoothness of $\tilde u_{i,k}$, the differentiability with respect to $\tilde Y_l$ follows by the smoothness of $\tilde j_{i,k,\underline r}$ and the fact that $\tilde Y_{l'}[\tilde Y_l, \tilde L_0] \tilde u_{i,k}=[\tilde Y_{l'},[\tilde Y_l, \tilde L_0] ] \tilde u_{i,k}+[\tilde Y_l, \tilde L_0]  \tilde v_{i,k,l',\underline r}$, since it is easy to check that, given any vector $Z$ tangent to $M^{|\Lambda|}$, $[\tilde Y_{l'},Z]$ is still tangent to $M$.} the Lemma follows by iterating the above argument.
\end{proof}

\begin{lem}\label{lem:l2loc} It holds true $L_*u_{i,k}\in L^2_{\text{loc}}(\bR^8,m_0)$.
\end{lem}
\begin{proof}
Recall the definition of $L_*$ given by (\ref{eq:Lstar}). Then all we
need to prove is that $\frac{\partial \Psi}{\partial p_i}$ is in
$L^2_{\text{loc}}(\bR^8,m_0)$. It follows from straightforward
calculations that the singularities at $0$ of  $\frac{\partial
  \Psi}{\partial p_i}$ are $m_0$ integrable.
\end{proof}

\section{Structure and regularity of $\gamma^2$ and $\alpha$}
\label{sec:regul-gamm-alpha}

Let $u = u_{1,2}$ the solution of the Poisson equation \eqref{poisson}. 
Then, by \eqref{eq:gamma} and \eqref{eq:canexpl}, we write 
\begin{equation}
  \label{eq:23}
  \begin{split}
    \gamma^2(a_1,a_2) &= \mu_{a_1,a_2} \left( j_{1,2} u_{1,2} \right) \\
   & = \frac{16\omega_4^2 a_1a_2}{ \cZ(a_1)\cZ(a_2)}
\int_{S^3\times S^3} d\sigma(\xi_1,\eta_1)d\sigma(\xi_2,\eta_2)\Omega(\sqrt{a},\xi,\eta). 
 \end{split}
\end{equation}
where $\sigma$ is the uniform probability measure on the sphere $S^3$ and
\begin{equation*}
\begin{split}
\Omega(r,\xi,\eta):=&\frac{(j_{1,2}u_{1,2})(\Psi^{-1}(r_1,\xi_1,\eta_1), \Psi^{-1}(r_2,\xi_2,\eta_2))}{\oU'(\rho(r_1^2\xi_1^2))\oU'(\rho(r_2^2\xi_2^2))}\\
=&\frac{(\tilde j_{1,2}\tilde u_{1,2})(r_1,\xi_1,\eta_1, r_2,\xi_2,\eta_2)}{\oU'(\rho(r_1^2\xi_1^2))\oU'(\rho(r_2^2\xi_2^2))}.
\end{split}
\end{equation*}

We have already seen that $\tilde u$ satisfies \eqref{eq:poisson-reg} and is a smooth function on $M^2$. Hence $\Omega\in\cC^\infty(M^2,\bR)$. By Remark \ref{rem:simmetry} and the structure of $\tilde L_0$ (see Lemma \ref{lem:regularize}), we have 
\[
(\tilde j_{1,2}\tilde u_{1,2})(-r_1,-\xi_1,-\eta_1, r_2,\xi_2,\eta_2)=(\tilde j_{1,2}\tilde u_{1,2})(r_1,\xi_1,\eta_1, r_2,\xi_2,\eta_2)
\]
and the same for the second coordinate. By the symmetry of the measure
$\sigma$ it follows then that the integral on the right hand side of
\eqref{eq:23} is an even smooth function of $\sqrt {a_1}, \sqrt{a_2}$,
hence a smooth function of $a_1,a_2$. This shows that
$\gamma^2\in\cC^\infty([0,\infty)^2,\bR)$. 

We are now in the position to prove the relation between $\alpha$ and $\gamma$:

\begin{lem}\label{lem:alphagamma}
For any nearest neighbor couple $\{i,k\}$:
\begin{equation}\label{F.2}
  e^{\mathcal U(\underline a)} 
    (\partial_{a_i}- \partial_{a_k} )  
   \left( e^{- \mathcal U(\underline a)}  \gamma^2(a_i,a_k)\right) 
 = \alpha(a_i,a_k).
\end{equation}
with $\mathcal U(\underline a) = - \sum_j \log {\mathcal Z}(a_j)$.
\end{lem}
\begin{proof}
By Lemma \ref{lem:int-part} follows
\[
\mu_a(g\partial_p f)=e^{-\mathcal U}\partial_a\left\{e^{\mathcal U}\mu_a(gpf)\right\}
\]
provided $g$ does not depend on $p$. Thus, since $u_{i,k}$ is locally bounded,
\begin{equation*}
  \begin{split}
    \alpha(a_i, a_k) &= \mu_{\underline a} \left( \nabla V(q_i
      -q_k) (\partial_{p_i} u_{i,k} - \partial_{p_k}u_{i,k})\right)\\
   & =e^{-\mathcal U} \partial_{a_i}\left\{e^{\mathcal U}
    \mu_{\underline a} \left( \nabla V(q_i-q_k) {p_i}  
      u_{i,k}\right)\right\}\\
   &\quad -e^{-\mathcal U} \partial_{a_k}\left\{e^{\mathcal U}
    \mu_{\underline a} \left( \nabla V(q_i-q_k) {p_k}  
      u_{i,k}\right)\right\}.
\end{split}
\end{equation*}
To continue, notice that for each smooth $\vf$,
\[
\begin{split}
\int& e^{-\sum_l a_l}\prod_l\mathcal Z(a_l) \vf(\underline{a})\mu_{\underline{a}} \left( p_i\nabla V(q_i
      -q_k)u_{i,k} \right)\\
      &=\bE_{m_0}\left(\vf(\underline{\mathcal E}^0) p_i\nabla V(q_i
      -q_k)u_{i,k} \right)\\
       &= -\bE_{m_0}\left(\vf(\underline{\mathcal E}^0)u_{i,k}
         L_0^* V(q_i 
      -q_k) \right)+\bE_{m_0}\left(\vf(\underline{\mathcal E}^0) p_k\nabla V(q_i
      -q_k)u_{i,k} \right)\\
       &= -\bE_{m_0}\left(\vf(\underline{\mathcal E}^0) V(q_i
      -q_k) L_0u_{i,k}\right) +\bE_{m_0}\left(\vf(\underline{\mathcal
        E}^0) p_k\nabla V(q_i 
      -q_k)u_{i,k} \right)\\
      &= -\bE_{m_0}\left(\vf(\underline{\mathcal E}^0) V(q_i
      -q_k) j_{i,k}\right)+\bE_{m_0}\left(\vf(\underline{\mathcal
        E}^0) p_k\nabla V(q_i 
      -q_k)u_{i,k} \right)\\
      &=\int e^{-\sum_l a_l}\prod_l\mathcal Z(a_l) \vf(\underline{a})\mu_{\underline{a}} \left( p_k\nabla V(q_i
      -q_k)u_{i,k} \right),
\end{split}
\]      
where we used the antisymmetry of $j_{i,k}$ with respect to $p$.

Hence 
\begin{equation}\label{eq:ha}
\mu_{\underline{a}} \left( p_i\nabla V(q_i-q_k)u_{i,k} \right)=\mu_{\underline{a}} \left( p_k\nabla V(q_i-q_k)u_{i,k} \right)
\end{equation}
$m_0$-alpmost surely.
Accordingly,
\begin{equation*}
  \begin{split}
    \alpha(a_i, a_k) &=\frac 12 e^{-\mathcal U}\left( \partial_{a_i}-\partial_{a_k}\right)\left\{e^{\mathcal U}
    \mu_{\underline a} \left( \nabla V(q_i-q_k) ({p_i}+p_k)  
      u_{i,k}\right)\right\}\\
      &= e^{-\mathcal U}\left( \partial_{a_i}-\partial_{a_k}\right)\left\{e^{\mathcal U}
    \mu_{\underline a} \left( j_{ik}  u_{i,k}\right)\right\}.\\
  \end{split}
\end{equation*}
The result follows since
\begin{equation*}
  \begin{split}
     \mu_{\underline a} \left( j_{i,k} u_{i,k}\right) &=
     \mu_{\underline a} 
    \left( u_{i,k} L_0 u_{i,k}\right) =   \sigma^2 \mu_{\underline a}
    \left( u_{i,k} S u_{i,k}\right) \\
    &=  - \sigma^2 \sum_{j=i,k} 
    \mu_{\underline a} ((X_{j} u_{i,k} ) ^2) = 
     -  \gamma^2(a_i,a_k).
  \end{split}
\end{equation*}

\end{proof}

We can rewrite \eqref{F.2} as
\begin{equation}\label{eq:alpha-form}
\alpha(a_i,a_k)=   (\partial_{a_i}- \partial_{a_k} )
\gamma^2(a_i,a_k) +\left(\frac{\cZ'(a_i)}{\cZ(a_i)} -
     \frac{\cZ'(a_k)}{\cZ(a_k)}\right)\gamma^2(a_i,a_k) . 
\end{equation}
Since $\frac{\cZ'(a)}{\cZ(a)} \sim a^{-1}$,
the regularity of $\alpha$ follows from the one of $\gamma^2$, if
we can prove that $\gamma^2 \sim a_1 a_2$. 

\begin{lem}\label{lem:regular-ag} There exists $G\in\cC^\infty([0,\infty)^2,\bR)$ such that, for each $a_1,a_2\geq 0$

\[
\begin{split}
\gamma^2(a_1,a_2)&=a_1a_2G(a_1,a_2)\geq 0,
\end{split}
\]
and furthermore $\alpha(0, a) \ge 0 \text{ for all } a \ge 0$.
\end{lem}
\begin{proof}
Observe that fixing the energy of the first particle $\mc E_1(q_1,p_1)
= 0$, it implies that $q_1=p_1=0$. So defining $\hat u(q,p) =
u_{1,2}(0,0,q,p)$, it solves on $\mathbb R^2$ the equation $L \hat u=
\hat \j$, where $L$ is the generator of the dynamics of a single
isolated atom and $\hat \j = - \frac 12 p\cdot \nabla V(q) = L^* V$. 
  
By the smoothness of  $\gamma^2$ it follows that
\[
\gamma^2(0,a)=-\mu_a(\hat \j \hat u)
=\mu_a(L^{*} V\cdot \hat u)=\mu_a(VL \hat u)=\mu_a(V\hat \j)=0
\]
due to the symmetry of $\mu_a$ with respect to the transformation
$p\to -p$. The structure of $\gamma^2$ follows then by the symmetry
and smoothness of $\gamma^2$. The positivity follows by 
\[
\gamma^2=-\mu_{a_1,a_2}(L_0u \cdot u)=\sigma^2\mu_{a_1,a_2}((X_1u)^2+(X_2u)^2).
\]
By \eqref{eq:alpha-form} and \eqref{eq:34}
\[
\alpha(a,0)= 2 aG(a,0) \ge 0.
\]
\end{proof}


\appendix
\section{Commutators}\label{sec:comm}
This appendix collects some formulae concerning commutators for the unperturbed system ($\ve=0$), we use the notation of Proposition \ref{lem:gap}.
In section \ref{sec:single-site-dynam} we have already computed:
\begin{equation}\label{eq:Com4}
[C_0,B] = \big(J p\cdot\partial_{q}+J \nabla U(q)\cdot\partial_{p}\big)
=:C_1.
\end{equation}
\begin{equation}\label{eq:Com5}
\begin{split}
[C_1,B] &=4 \oU'(q)Jq \cdot \partial_q - 4\{\oU''(q)|q|^2+\oU'\}  Jp \cdot \partial_p\\
&=:2C_2 + \zeta(q) C_0=:2C_2+R_2,
\end{split}
\end{equation}
An explicit computation shows that
\begin{equation}\label{eq:c2-bound}
C_2=\frac{p\cdot J \nabla U}{p^2}B+\frac{p\cdot\nabla U}{p^2}C_1-\frac{\|\nabla U\|^2}{p^2}C_0.
\end{equation}
The above formula shows that $C_2$, for small $p$ and large $q$, is not bounded by $B, C_0, C_1$ and, as we will see in the following, this forces us to compute
\[
\begin{split}
R_3 = [C_2,B] &= -\big(J \nabla U(q)\cdot D^2U(q) \partial_{p} + 
 J D^2U(q) p\cdot \partial_{q}\big)\\
 &=-\{4\oU''\langle p,q\rangle Jq+2\oU' Jp\}\partial_q-4(\oU')^2Jq\partial_p.
\end{split}
\]
By using the orthonormal base $Z_0, Z_1, Z_2$, defined in \eqref{eq:ort-base}, we have that 
\begin{equation*}
  \begin{split}
    R_3& = \sum_{j=0}^2 \langle R_3, Z_j \rangle Z_j \\
   & = \mathcal N^{-2}  \langle R_3 ,B \rangle B + \mathcal N^{-2}  \langle R_3,
    C_1 \rangle C_1 + \mathcal N^{-2}  \langle R_3 , (C_2 - C_0) \rangle (C_2 - C_0)\\
    &=4\cN^{-2}\oU''\langle q,p\rangle\langle Jp,q\rangle B -\cN^{-2}\left[4\oU''\langle q,p\rangle^2+2\oU'p^2+8(\oU')^3q^2\right]C_1\\
    &\quad - 8\cN^{-2}\oU''\oU'\langle q,p\rangle q^2(C_2 - C_0)\,,
  \end{split}
\end{equation*}
and an explicit calculation shows that
\begin{equation}\label{eq:r3-bound}
\|R_3\|^2=  \mathcal N^{-2} \left(| \langle R_3 , B \rangle|^2  + | \langle R_3 ,
    C_1 \rangle|^2 + | \langle R_3, (C_2 - C_0) \rangle|^2\right) \le K
\end{equation}

To conclude the {\em first order} analysis we need to compute some more commutators
\begin{equation}\label{eq:Com7}
\begin{split}
&([C_0,C_1])=-B\\
&([C_0,C_2])=0\\
&([C_1,C_2])=:\rho C_2-\beta B,
\end{split}
\end{equation}
with $\rho$ and $\beta$ also bounded.

We also need some \emph{second order} commutators:
\begin{equation}
  \label{eq:Com72}
  \begin{split}
    &[B,C_0^2]=C_0BC_0-C_0^2B-C_1C_0=-2C_1C_0+B\\
&[C_1,C_0^2]=C_0B+BC_0=2C_0B-C_1.
  \end{split}
\end{equation}

\section{Hypocoercivity estimates}\label{sec:hypoest}

This appendix contains the core of the hypocoercivity argument. For our purposes it turns out to be more convenient to set it in an abstract setting.

Let $\cH^0$ be an Hilbert space and $C_0, B$ be closed operators satisfying the relations \eqref{eq:Com4}--\eqref{eq:Com72}. Assume that $C_k, B, C_kC_0, BC_0$ have all a common core $D_c$. In addition, assume that for each $h,g\in D_c$ and $Z\in\{C_0,B\}$,
\begin{equation}\label{eq:antisimmetric}
\langle h, Z g\rangle=-\langle Zh, g\rangle,
\end{equation}
where $\langle\cdot,\cdot\rangle $ is the scalar product of $\cH^0$.
Remark that the various constant that will appear in the results ($\tau, K$
etc.) do not depend on the scalar product  $\langle h,
g\rangle$.

We are interested in obtaining estimates in terms of the following generalized Sobolev
norm: 
\[
\|h\|_{\scH^1}^2:=\|h\|^2+\sum_{k=0}^2\|C_kh\|^2 +\|Bh\|^2.
\]
To this end it turns out to be convenient to define the bilinear form:
\begin{equation}\label{eq:twoangles}
\begin{split}
\lv h,g\rv:=&\langle h,g\rangle+\sum_{k=0}^2a_k\langle C_kh, C_k
g\rangle +a_3\langle Bh, Bg\rangle \\
& -b_0\langle C_0h,C_1g\rangle-b_0\langle C_1h,C_0g\rangle
-b_1\langle C_1h,C_2g\rangle-b_1\langle C_2h,C_1g\rangle
\end{split}
\end{equation}
where $a_k>0,b_k>0$ will be chosen shortly (see (\ref{eq:pos2}) and
(\ref{eq:choice-param})). 

\begin{lem}\label{lem:positive}
If for $\delta\in (0,1)$
\begin{equation}
a_0b_1^2+a_2b_0^2 \leq a_0a_1a_2(1-\delta)^2, \label{eq:condpos}
\end{equation}
 then the quadratic form is positive definite and
\[
\lv h,h\rv\geq  (\|h\|^2+ \delta \sum_{k=0}^2 a_k \|C_kh\|^2
+ a_3 \|Bh\|^2)= \kappa \|h\|_{\scH^1}^2. 
\]
with $\kappa = \min\{\delta a_0, \delta a_1, \delta a_2, a_3, 1\}$.
\end{lem}
\begin{proof}
We have for any $\alpha_0, \alpha_1 >0$
\[
\begin{split}
\lv h,h\rv& - a_3 \|Bh\|^2- \|h\|^2\geq 
\sum_{k=0}^2a_k\|C_kh\|^2
- \sum_{k=0}^1 \left(b_k\alpha_k \|C_kh\|^2 + b_k
  \alpha_k^{-1}\|C_{k+1}h\|^2\right) \\ 
=& \left(a_0 - b_0 \alpha_0\right) \|C_0h\|^2
+ \left(a_1 -  b_0 \alpha_0^{-1} -  b_1\alpha_1^{-1}\right) \|C_1h\|^2
+ \left(a_2 - b_1 \alpha_1\right) \|C_2h\|^2\\
=& \delta \sum_{k=0}^2a_k\|C_kh\|^2 
+ \left(a_0(1-\delta) -  b_0 \alpha_0\right) \|C_0h\|^2
+ \left(a_2(1-\delta) -  b_1 \alpha_1\right) \|C_2h\|^2\\
&+ \left(a_1(1-\delta) -  b_0 \alpha_0^{-1} -  b_1
  \alpha_1^{-1}\right) \|C_1h\|^2
\end{split}
\]
Then choosing
\begin{equation*}
  \alpha_0 = \frac {a_0 (1-\delta)}{b_0}, \qquad 
   \alpha_1 = \frac {a_2 (1-\delta)}{b_1}
\end{equation*}
the Lemma follows immediately by condition (\ref{eq:condpos}).
\end{proof}

On the other hand Schwartz inequality implies that there exists $K>0$ such that
\begin{equation}\label{eq:norm-eq}
\lv h,h\rv\leq K (\|h\|^2+\sum_{k=0}^2\|C_kh\|^2 +\|Bh\|^2).
\end{equation} 

 Let $\cH^1:=\{h\in\cH^0\;:\; \lv h, h\rv<\infty\}$, clearly it is an Hilbert space, equivalent to $\scH^1$, with scalar product $\lv\cdot,\cdot\rv$.

 \begin{lem}
   \label{lem:hypo} Given $C_0, B$ as described at the beginning of the section, there exists $\tau>0$ such that, for each $\sigma\in(0,1)$, 
   \begin{equation}\label{eq:lower}
\lv h,Lh\rv\leq -\tau\sigma^2\left\{\sum_{k=0}^2\|C_kh\|^2+\sum_{k=0}^2\|C_kC_0h\|^2+\|Bh\|^2+
\|BC_0h\|^2\right\}.
\end{equation}
 \end{lem}
\begin{proof}
This is a proof by boring computations. Let us start:
\[
\begin{split}
\lv h,Lh\rv=&\sigma^2\langle h, C_0^2 h\rangle +
\bigg\{ \sum_{k=0}^2 a_k[\sigma^2\underbrace{\langle
  C_kC_0^2h,C_kh\rangle}_{{\bf I_{A,k}}}+\underbrace{\langle
  C_kBh,C_kh\rangle}_{{\bf I_{B,k}}}]\\ 
&+\sigma^2 a_3\underbrace{\langle BC_0^2h,Bh\rangle}_{{\bf I_{B,B}}}
-\sum_{k=0}^1\sigma^2 b_k\underbrace{\left(\langle C_kh,C_{k+1}C_0^2h\rangle
+\langle C_kC_0^2h,C_{k+1}h\rangle\right)}_{{\bf II_{A,k}}}\\
&-\sum_{k=0}^1 b_k\underbrace{\left(\langle C_kh,C_{k+1}Bh\rangle+
\langle C_kBh,C_{k+1}h\rangle\right)}_{{\bf II_{B,k}}}\bigg\}.
\end{split}
\]
Now we must look at all the terms one by one, we will use systematically \eqref{eq:Com4}-- \eqref{eq:Com7}.
\[
\langle h, C_0^2 h\rangle=-\langle C_0 h, C_0 h\rangle.
\]
\[
\tag{$\bf I_{A,0}$}
\langle C_0^3h,C_0h\rangle=-\langle C_0^2h,C_0^2h\rangle.
\]

\begin{align}
 \langle C_1C_0^2h,C_1h\rangle&=\langle C_0C_1C_0h,C_1h\rangle+\langle BC_0h,C_1h\rangle\notag\\
&=-\langle C_1C_0h,C_0C_1h\rangle+\langle C_0 Bh,C_1h\rangle-\langle C_1h,C_1h\rangle   \tag{$\bf I_{A,1}$}\\
&=-\langle C_1C_0h,C_1C_0h\rangle+\langle C_1C_0h,Bh\rangle-\langle Bh,C_0C_1h\rangle\notag\\
&\quad-\langle C_1h,C_1h\rangle\notag\\
&=-\langle C_1C_0h,C_1C_0h\rangle+\langle Bh,Bh\rangle-\langle C_1h,C_1h\rangle.\notag
\end{align}

\[
\tag{$\bf I_{A,2}$}
\langle C_2C_0^2h,C_2h\rangle=-\langle C_2C_0h,C_2C_0h \rangle.
\]

\[
\tag{$\bf I_{B,0}$}
\langle C_0 Bh,C_0h\rangle=\langle C_1h,C_0h\rangle.
\]

\[
\tag{$\bf I_{B,1}$}
\langle C_1 Bh,C_1h\rangle=2\langle C_2h,C_1h\rangle+\langle R_2h,C_1h\rangle .
\]

\[
\tag{$\bf I_{B,2}$}
\langle C_2 Bh,C_2h\rangle=\langle R_3h,C_2h\rangle .
\]

\begin{align}
\langle BC_0^2h,Bh\rangle&=-\langle BC_0h,C_0Bh\rangle-\langle C_1C_0h,Bh\rangle\notag\\
&=-\langle BC_0h,BC_0h\rangle -\langle BC_0 h,C_1h\rangle-\langle C_1C_0h,Bh\rangle  \tag{$\bf I_{B,B}$}\\
&=-\langle BC_0 h,BC_0 h\rangle +\langle Bh,C_0C_1h\rangle +\langle C_1h,C_1h\rangle-\langle C_1C_0h,Bh\rangle \notag\\
&=-\langle BC_0 h,BC_0 h\rangle -\langle Bh,Bh\rangle +\langle C_1h,C_1h\rangle .\notag
\end{align}

\[
\tag{$\bf II_{A,0}$}
\langle C_0 C_0^2h,C_1h\rangle+\langle C_0h,C_1 C_0^2h\rangle=-2\langle C_0^2 h,C_1C_0h\rangle-\langle C_0h, C_1h\rangle .
\]

\[
\tag{$\bf II_{A,1}$}
\langle C_1 C_0^2h,C_2 h\rangle+\langle C_1h,C_2C_0^2h\rangle=-2\langle C_1 C_0h,C_2C_0h\rangle
-\langle C_1 h,C_2h\rangle .
\]

\[
\tag{$\bf II_{B,0}$}
\langle C_0 Bh,C_1h\rangle+\langle C_0h,C_1 Bh\rangle=\langle C_1 h,C_1h\rangle +2\langle C_0h,C_2h \rangle +\langle C_0 h,R_2h\rangle.
\]

\[
\tag{$\bf II_{B,1}$}
\langle C_1 Bh,C_2 h\rangle+\langle C_1h,C_2 Bh\rangle=2\langle C_2h,C_2 h\rangle+\langle R_2 h,C_2h\rangle +
\langle C_1 h,R_3h\rangle .
\]
Finally, we can put all the terms together, obtaining 
\begin{equation*}
  \begin{split}
    \lv h,Lh\rv = &-\sigma^2 \|C_0 h\|^2 
    - \left[b_0 +\sigma^2 (a_1- a_3)\right] \|C_1 h\|^2 
    - 2b_1 \|C_2 h\|^2\\
    &-\sigma^2 (a_3-a_1) \|B h\|^2
   + \left(a_0 + \sigma^2b_0\right)\langle C_0h, C_1h\rangle
   + a_1\langle R_2 h, C_1h\rangle\\
  & -b_0 \langle R_2 h, C_0h\rangle-b_1\langle R_2 h, C_2h\rangle+ (2a_1+ \sigma^2 b_1) \langle C_1h, C_2 h\rangle \\
   & -2b_0  \langle C_0h, C_2 h\rangle +a_2  \langle R_3h, C_2
   h\rangle - b_1  \langle R_3 h, C_1 h\rangle\\
   &- \underbrace{ \sigma^2
     \left( \sum_{k=0}^2 a_k \|C_kC_0h\|^2 + a_3\|BC_0h\|^2 
     - 2 \sum_{k=0}^1 b_k \langle C_k C_0h, C_{k+1} C_0h \rangle
   \right)}_{ \bf III}
  \end{split}
\end{equation*}
By the same argument used in the proof of Lemma \ref{lem:positive},
if for $0<\delta<1$ we have
\begin{equation}
a_0b_1^2+a_2b_0^2 < a_0a_1a_2(1-\delta)^2, 
\label{eq:pos2}
\end{equation}
 then {\bf III} is bounded by $-\sigma^2 \delta\left(\sum_k a_k\|C_k C_0 h\|^2 +a_3\|BC_0h\|^2\right)$.

Recalling \eqref{eq:Com5} and \eqref{eq:r3-bound}, we have $\|R_2h\|^2 \le K_2 \|C_0 h\|^2$,
 $\|R_3 h\|^2\leq K_3\{\|Bh\|^2+\|C_0h\|^2+\|C_1h\|^2+\|C_2h\|^2\}$.
Applying Schwarz, for each $\alpha_1,\alpha_2,\alpha_3, \alpha_4>0$, the first three lines of the above equation are bounded by
\begin{equation*}
  \begin{split}
    &-\left\{\sigma^2 - \frac 12\alpha_1[a_0+\sigma^2 b_0 + K_2 a_1] -K_2 (b_0+b_1)- \alpha_3 b_0 - \frac 12
        K_3 a_2 - \frac 12 b_1 K_3\alpha_4\right\}\|C_0h\|^2\\ 
   & -\bigg\{b_0 - \sigma^2 (a_3- a_1) - \frac {a_0+\sigma^2 b_0 +a_1}{2\alpha_1} -
       \frac {2a_1+ \sigma^2 b_1}{2\alpha_2} -  \frac {K_3a_2}2  -  \frac {(K_3\alpha_4^2 +1) b_1}{2\alpha_4}  \bigg\}
       \|C_1 h\|^2 \\ 
    &-\left\{ 2b_1 -\frac 12 b_1-  \frac 12\alpha_2(2a_1+ \sigma^2 b_1) - \alpha_3^{-1} b_0 -  \frac 12 (K_3 +
      1) a_2 -\frac 12 b_1 K_3\alpha_4 \right\} \|C_2 h\|^2 \\
    & - \left\{\sigma^2 (a_3 - a_1) -\frac 12 K_3 a_2 
     - \frac 12 K_3 b_1\alpha_4\right\} \|Bh\|^2 
  \end{split}
\end{equation*}
An inspection of the above expression shows that with the choices
\begin{equation}\label{eq:choice-param}
a_0=\sigma^2\upsilon^7\,;\; a_1=\sigma^2\upsilon^{16}\,;\; a_2=\sigma^2\upsilon^{18}\,;\;
a_3=\upsilon^{13}\,;\; b_0=\sigma^2 \upsilon^{12}\,;\; 
b_1= \sigma^2 \upsilon^{17}
\end{equation}
and
\begin{equation*}
  \alpha_1 = \upsilon^{-6}\,;\; \alpha_2 = \upsilon^2 \,;\; 
  \alpha_3 = \upsilon^{-5}\;;\; \alpha_4=\upsilon
\end{equation*}
implies that, by choosing $\upsilon$ small enough, there exist $\tau> 0$ such that, for each $\sigma \in(0, 1)$, 
\begin{equation}\label{eq:lower-gap}
\lv h,Lh\rv\leq -\tau\sigma^2\left\{\sum_{k=0}^2\|C_kh\|^2 +\|Bh\|^2
+\sum_{k=0}^2\|C_kC_0h\|^2 + \|BC_0h\|^2\right\}.
\end{equation}
Observe that (\ref{eq:pos2}) is satisfied by this choice.
\end{proof}

\section{A coordinate change}\label{sec:coordinate}

We study a change of coordinates in the case on one particle, for
many particles one simply considers the product. 

Let $S^3:=\{x\in\bR^4\;:\;\|x\|=1\}$ and $M=\bR\times S^3\subset \bR\times \bR^2\times\bR^2$, $M_+=(0,\infty)\times S^3\subset M$. Let $\Psi: \bR^4\setminus\{0\}\to M_+\subset \bR^5$ be defined by
\[
\Psi(q,p)=\left(\sqrt{\frac{p^2}2+U(q)}, 
\frac{q\sqrt{U(q)}}{|q|\sqrt{\frac{p^2}2+U(q)}},
\frac{p}{\sqrt{{p^2} + 2U(q)}}\right)=:(r,\xi,\eta)
\]
Remember that, by hypotheses, $U(q)=\oU(q^2)$ and $\oU$ is a strictly increasing function. We can then extend $\oU$ to a smooth increasing function  on $\bR$ such that $\oU(z)\geq 0$ if $z\geq 0$.\footnote{Clearly the extension is arbitrary, but this is irrelevant in the following.} It follows that $\rho(z):=\oU^{-1}(z)$ is a well defined smooth function on $\bR$ such that $\rho(0)=0$.

One can readily check that the inverse $\Psi^{-1}: M_+\to\bR^4$ is given by
\begin{equation}\label{eq:35}
\Psi^{-1}(r,\xi,\eta)=\left(\sqrt{\rho(r^2\xi^2)}
\frac{\xi}{\|\xi\|}, \sqrt {2}\,r\eta\right).
\end{equation}

For the following it is convenient to introduce the function $\theta(z)=\sqrt{\rho(z)/z}$, $z\neq 0$. Notice that $\theta$ is smooth on $\bR$ provided we set $\theta(0)=\sqrt{1/\oU'(0)}$.

Next we transport the vector fields on $M$ by the usual formula $\Psi_*Z=(D\Psi Z)\circ \Psi^{-1}$.
The following lemma follows by the computation of $D\Psi$ that can be found in appendix \ref{sec:micr-meas}:

\begin{lem}\label{lem:regularize}
With the above notations we have
\[
\begin{split}
&\tilde C_0:=\Psi_*C_0= J\eta\cdot\partial_{\eta} \\
&\tilde B:=\Psi_*B=\frac{\sqrt 2}{\theta(r^2\xi^2)}\left\{\eta- \xi\frac{\langle\xi,\eta\rangle}{\xi^2} \left[1- 
\oU'(\rho(r^2\xi^2))
\theta(r^2\xi^2)^2\right]\right\}\partial_\xi \\
&\quad\quad\quad\quad\quad\quad-\sqrt 2 \oU'(\rho(r^2\xi^2))\theta(r^2\xi^2)\xi\partial_{\eta}\\
&\tilde Y:=\Psi_* Y= \partial_r.
\end{split}
\]
\end{lem}

Define the function
\[
\Gamma(z):= 1-\oU'(\rho(z))\theta(z)^2
\]
and notice that it is continuous in $0$ and $\Gamma(0) = 0$. 
It is then natural to write
\begin{equation}\label{eq:btilde}
\begin{split}
&\tilde B = \frac{\sqrt 2}{\theta(r^2\xi^2)} \left\{\tilde B^0 +
\Gamma(r^2\xi^2)\tilde B^1 \right\}\\
&\tilde B^0=\eta\partial_\xi-\xi\partial_\eta\,,\quad \tilde B^1=\xi\partial_\eta - \frac{\langle\xi,\eta\rangle}{\xi^2}  \xi\partial_{\xi}
\end{split}
\end{equation}
Notice that $\theta(r^2\xi^2)$ and $\Gamma(r^2\xi^2)$ are smooth function of $\xi$ and $r$, consequently $\tilde B$ is a smooth vector field. In addition, both $\tilde B^0$ and $\tilde B^1$ are tangent to the surfaces $\{r\}\times S^3\subset M$.
\begin{rem} The vector fields $\tilde C_0, \tilde B,\tilde Y$ are defined only on $M_+$ but admit a smooth canonical extension to all $M$. We will use the same notation to designate such an extension.
\end{rem}

It follows that there exists smooth vector fields $\tilde C_1, \tilde C_2$ on $M$ such that $\tilde C_i=\Psi_*C_i$ on $M_+$.

\section{Microcanonical measure}
\label{sec:micr-meas}

We collect here some properties about the microcanonical measure 
$\mu_a$ (and its product version $\mu_{\underline a}$).

Let us first recall the definition and some formulas. Microcanonical measure is defined as the conditional measure on the energy shell 
$\Sigma_a = \{(q,p)\in \bR^4 : \|p\|^2/2 + U(q) = a \}$. This means that for every continuous functions $\phi:\bR_+ \to \bR$ and $f$ on $\bR^4$ we have
\begin{equation}
  \label{eq:30}
  \int \phi(\mathcal E_0(p,q)) f(p,q) dm_0 = 
  \int_{\bR_+} \phi(a)  \mu_a(f) e^{-a} \mathcal Z(a)  da
\end{equation}
Standard formulas give that 
\begin{equation}
  \label{eq:31}
  \mathcal Z(a) = \int_{\Sigma_a} \left( p^2 + U'(q)^2\right)^{-1/2}    
  d\sigma_{\Sigma_a}(p,q) 
\end{equation}
and
\begin{equation}
  \label{eq:32}
   \mu_a(f) =  \mathcal Z(a)^{-1} \int_{\Sigma_a} f(p,q) 
   \left( p^2 + U'(q)^2\right)^{-1/2} d\sigma_{\Sigma_a}(p,q) 
\end{equation}
where $\sigma_{\Sigma_a}$ is the Lebesgue measure on the energy shell.

We need an integration by part formula for the microcanonical measure $\mu_a$. This is provided by the following proposition. 

\begin{prop}\label{lem:int-part}
For all $f,g$ continuous locally bounded functions on $\bR^4$,
 such that $g$ does not depend on
$p_k$ ($\partial_{p_k}g\equiv 0$), and $f$ is differentiable in $p_k$, then 
\begin{equation}
\label{intparmic}
\mu_{\underline a}
(g\partial_{p_k}f)=(\partial_{a_k}\mu_{\underline a})(g \,p_kf)
+ \frac {\mathcal Z'(a_k)}{\mathcal Z(a_k)} 
\mu_{\underline a}(g \,p_kf)   ,
\qquad m_{0,\beta}\text{-a.s.},\; a_k>0. 
\end{equation}
\end{prop}
\begin{proof}
Because of the product structure of $\mu_{\underline {a}}$,
we can just consider only one site and we drop the index $k$ of it.

Let us recall that from the definition of $\mu_{a}$
follows, for each  smooth bounded functions $f,g$ and $\vf:\bR_+\to\bR$ supported away from zero,
\[
\int_{\bR^4}\; dp\; dq\;  e^{- H_0(p,q)}
\; \vf(H_0)\; f(p,q)\ =\ \int_{\bR_+}  \, \;da\; e^{-a} \vf(a) 
\mathcal Z(a) \mu_{a}(f)  
\]
Thus
\[
\begin{split}
\int_{\bR_+} \;da\; e^{-a}
\vf(a) \mathcal Z(a) \mu_{a}(g\partial_p f)=  
\int_{\bR^4}\; dp\; dq\;  e^{- H_0}
\vf(H_0) g(q) \partial_{p}f (p,q)\\ 
=\int_{\bR^4}\; dp\; dq\;  e^{- H_0} 
\left\{-\vf'(H_0) + \vf(H_0)\right\} g(q)
p f(p,q)\\  
=\int_{\bR_+}  \;da \; e^{-a} 
\left\{-\vf'(a)+ \vf(a)\right\} \mathcal Z(a) \mu_{a}(g pf) \\ 
=\int_{\bR_+}  \; da \; e^{-a} 
\vf(a) \partial_{a}\left( \mathcal Z(a) \mu_{a}(g p f)\right) \\
=\int_{\bR_+}  \; da \; e^{-a}  
\vf(a) \mathcal Z(a)\partial_{a}\left( \mu_{a}(g p f)\right)
+ \int_{\bR_+}  \; da \; e^{-a} 
\vf(a) \frac{\mathcal Z'(a)}{\mathcal Z(a)} \mathcal Z(a) 
\mu_{a}(g p f).
\end{split}
\]
It follows the relation \eqref{intparmic} for any bounded smooth $f,g$. The result follows by approximations.
\end{proof}

Formula (\ref{eq:32}) is difficult to be used directly, but exploiting the symmetry and the convexity of the potential, it is possible to write this microcanonical expectation as as integral on the 3 dimensional sphere of radius 1 with respect the corresponding uniform measure.
In fact the strict convexity of $U$ makes the energy shell very close to a sphere for small energy $a$, and $\mu_a$ close to the uniform measure on this sphere. We want to study this more precisely.

Recall the change of coordinates $\Psi$ introduced in Section
\ref{sec:regular}, by (\ref{eq:regularize}) and its inverse (\ref{eq:35}). 
Recall also the notation $\theta(z)=\sqrt{\rho(z)/z}$, and that $\theta(0) \to \sqrt{1/\oU'(0)}$.

\begin{lem}
  \begin{equation}
    \label{eq:34}
    \mathcal Z(a)= 4\omega_4 a\int_{S^3} \left[\oU'(\rho(a\xi^2))\right]^{-1}\,d\sigma(\xi,\eta) , 
  \end{equation}
where we have used polar coordinates in four dimensions, $\sigma$ is
the uniform probability measure on $S^3$, the unit four dimensional
ball, and $\omega_4$ is its volume. Furthermore
\begin{equation}\label{eq:canexpl}
\mu_a(f)= \frac{4\omega_4 a}{ \mathcal Z(a)}
\int_{S^3} f\circ
\Psi^{-1}(a,\xi,\eta) \left[\oU'(\rho(\sqrt{a}\xi^2))\right]^{-1}
\,d\sigma(\xi,\eta). 
\end{equation}
\end{lem}

\begin{proof}
Instead of computing with differential forms it turns out to be more
efficient to use the following trick.  

Consider the change of variables $\tilde \Psi:(\bR^4\setminus
\{0\})\times \bR_+\to\bR^5$ defined by 
\begin{equation}\label{eq:regularizeplus}
\tilde\Psi(q,p, s)=\left(\frac{p^2}2+U(q), 
\frac{s\,q\sqrt{U(q)}}{|q|\sqrt{\frac{p^2}2+U(q)}},
\frac{s\,p}{\sqrt{{p^2} + 2U(q)}}\right)
\end{equation}
Note that, $\tilde \Psi$ is invertible and, setting $\varrho(\tilde \xi,\tilde \eta)=\sqrt{\tilde \xi^2+\tilde \eta^2}$,
\[
\tilde\Psi^{-1}(a,\tilde \xi,\tilde \eta)=\left(\sqrt{\rho(a\varrho^{-2}\tilde \xi^2)}
\frac{\tilde \xi}{\|\tilde \xi\|}, \varrho^{-1}\sqrt {2 a}\,\tilde \eta, \varrho\right).
\]
Then, given any two compact support functions $g\in\cC^0(\bR_+,\bR)$, $f\in\cC^0(\bR^4,\bR)$, we can write
\[
\begin{split}
\int_{\bR^5}& g(s)f(p,q) \;dq\,dp\, ds=\int_{\bR^5} \left[g f\left|\det(D\tilde \Psi)\right|^{-1}\right]\circ \tilde\Psi^{-1}(a,\tilde\xi,\tilde\eta) da\,d\tilde\xi\,d\tilde\eta\\
&=4\omega_4\int_{\bR^2\times S^3} g(s)\left[ f\left|\det(D\tilde \Psi)\right|^{-1}\right]\circ \tilde\Psi^{-1}(a,s\xi,s\eta)s^{3} da\,ds\,d\sigma(\xi,\eta)
\end{split}
\]

In Lemma \ref{lem:28} we compute the determinant of
\begin{equation}
  \label{eq:26}
D\tilde \Psi\circ\tilde\Psi^{-1}(a,\tilde \xi,\tilde \eta)= \begin{pmatrix}
\frac{\partial a}{\partial q} &\frac{\partial a}{\partial p} & \frac{\partial a}{\partial s}\\
\frac{\partial \tilde\xi}{\partial q}& \frac{\partial\tilde \xi}{\partial p} & \frac{\partial \tilde \xi}{\partial p}\\
\frac{\partial \tilde\eta}{\partial q}&\frac{\partial \tilde\eta}{\partial p}
& \frac{\partial \tilde \eta}{\partial p}
\end{pmatrix}\circ \tilde\Psi^{-1}(a,\tilde \xi,\tilde \eta)
\end{equation}
for $\tilde \xi^2+\tilde\eta^2=1$,
obtaining 
\begin{equation}
  \label{eq:28}
 \det\left( D\tilde \Psi\circ\tilde\Psi^{-1}\right) = \frac {\tilde U'(\rho(a\xi^2))}{a}.
 \end{equation}

 Thus, if we take a sequence of $g_n$ that converges weakly to the delta function on one, we have the formula
\[
\int_{\bR^4}f(p,q) \;dq\,dp=4\omega_4\int_{M_+} f\circ
\tilde\Psi^{-1}(a,\xi,\eta) \frac {a}{\tilde U'(\rho(a\xi^2))}\;
da\,d\sigma(\xi,\eta). 
\]
Accordingly, for each $g\in L^1(\bR_+)$ we can write
\[
\int_{\bR^4}f(p,q)g(\cE^0) \;dq\,dp=4\omega_4\int_{M_+} g(a)  f\circ
\tilde\Psi^{-1}(a,\xi,\eta) \frac {a}{\tilde U'(\rho(a\xi^2))}\;
da\,d\sigma(\xi,\eta) 
\]
On the other hand by the definition of the microcanonical measure $\mu_a$:
\[
\int_{\bR^4}f(p,q)g(\cE^0) \;dq\,dp=\int_{\bR_+}g(a) \mathcal Z(a)
\mu_a(f) da, 
\]
The above, by the arbitrariness of $g$, yields the following formula for the microcanonical expectation:
\begin{equation}\label{eq:compute}
\mu_a(f) = 4\omega_4 a \mathcal Z(a)^{-1}\int_{S^3} f\circ
\tilde\Psi^{-1}(a,\xi,\eta) \left[\tilde U'(\rho(a\xi^2))\right]^{-1}
\,d\sigma(\xi,\eta). 
\end{equation}

Putting $f=1$ in \eqref{eq:compute} implies the formula for $\mathcal
Z(a)$.
\end{proof}

\begin{lem}\label{lem:28}
  Proof of equation \eqref{eq:28}.
\end{lem}

\begin{proof}
An explicit computation of the derivative yields 
\[
D\tilde \Psi\circ\tilde\Psi^{-1}=
\begin{pmatrix}
2\tilde U'(\rho(a\xi^2\varrho^{-2}))\theta(a\xi^2\varrho^{-2})\sqrt{a}\varrho^{-1} \xi&\varrho^{-1} \sqrt{2a}\eta& 0 \\ \\
\frac{\varrho}{\sqrt a\theta}
\left[\Id-\frac{\xi\otimes\xi}{\xi^2}\{1- \eta^2\varrho^{-2}\tilde U'(\rho)\theta^2\right]
& -\frac{1}{\sqrt{2a}\,\varrho}\xi\otimes \eta& \varrho^{-1}\xi^t\\ \\
-\frac{1}{\sqrt{a}\,\varrho}\theta\tilde U'(\rho) \eta\otimes \xi&
\frac{\varrho}{\sqrt{ 2a}}\left[\Id-\varrho^{-2}\eta\otimes \eta\right] &\varrho^{-1}\eta^t
\end{pmatrix}.
\]
We want to compute the determinant for $(\xi,\eta)\in S^3$, i.e. $\varrho=1$. If we multiply the first row by $(2a)^{-1}\xi_1$ and we sum it to the second, then by $(2a)^{-1}\xi_2$ and sum it to the third, by $(2a)^{-1}\eta_1$ and sum to the fourth and finally by $(2a)^{-1}\eta_2$ and sum it to the last row, we have
\[
\begin{split}
\det\left(D\tilde \Psi\circ\tilde\Psi^{-1}\right) &=
\det\begin{pmatrix}
2\tilde U'(\rho(a\xi^2))\theta(a\xi^2)\sqrt{a} \xi& \sqrt{2a}\eta& 0 \\ \\
\frac{1}{\sqrt a\theta}
\left[\Id-\frac{\xi\otimes\xi}{\xi^2}\{1- \tilde U'(\rho)\theta^2\right]
& 0& \xi^t\\ \\
0&\Id(2a)^{-\frac 12}&\eta^t
\end{pmatrix}\\ &\\
&=\det\begin{pmatrix}
2\tilde U'(\rho(a\xi^2))\theta(a\xi^2)\sqrt{a} \xi& 0& -2a\eta^2\\ \\
\frac{1}{\sqrt a\theta}
\left[\Id-\frac{\xi\otimes\xi}{\xi^2}\{1- \tilde U'(\rho)\theta^2\right]
& 0& \xi^t\\ \\
0&\Id(2a)^{-\frac 12}&\eta^t
\end{pmatrix}\\ &\\
&=\frac 1{2a}\det\begin{pmatrix}
2\tilde U'(\rho(a\xi^2))\theta(a\xi^2)\sqrt{a} \xi& -2a\eta^2\\ \\
\frac{1}{\sqrt a\theta}
\left[\Id-\frac{\xi\otimes\xi}{\xi^2}\{1- \tilde U'(\rho)\theta^2\right]
& \xi^t\end{pmatrix}\\ &\\
&=\frac 1{2a}\det\begin{pmatrix}
0& -2a\eta^2\\ \\
\frac{1}{\sqrt a\theta}
\left[\Id-\frac{\xi\otimes\xi}{\xi^2}\{1- \eta^{-2}\tilde U'(\rho)\theta^2\right]
& \xi^t\end{pmatrix}.
\end{split}
\]
From the above the Lemma follows.\footnote{Just remember that $\det(\Id-bv\otimes v)=1-bv^2$ since $v$ and any vector perpendicular to $v$ are eigenvectors of eigenvalue $1-bv^2$ and $1$ respectively.}
\end{proof}

\begin{cor}
  Let $f$ be a continuous function of $(q,p)$, then the following formula holds
  \begin{equation}
    \label{eq:25}
    \mu_a( f(q,p) ) = \int_{S^3} f\circ\Psi^{-1}(a,\xi,\eta) 
             \; d\sigma(\xi,\eta)  + \cO(a)
  \end{equation}
where $\sigma$ is the uniform probability measure on $S^3$, and
$\cO(a)$ is a smooth function of order $a$ as $a\to 0$.
\end{cor}

Notice that because of (\ref{eq:compute}), since $0< c \le \tilde U'\le  c^{-1} < +\infty$, microcanonical measure is uniformly equivalent to the uniform measure on the unit sphere, and, for any $a> 0$, and any positive $f$:
\begin{equation}
  \label{eq:27}
  c^{-2} \int_{S^3} f\circ \tilde \Psi^{-1} (a,\xi,\eta)\; 
  d\sigma(\xi,\eta) \le \mu_a(f) 
  \le c^2 \int_{S^3} f\circ \tilde \Psi^{-1} (a,\xi,\eta) \; 
  d\sigma(\xi,\eta)
\end{equation}

\section{The Gaussian case}
\label{app:gauss}
Just to give a concrete idea of what we are doing and to provide some concrete intuition, here we discuss the case in which we have just two particles and both $U$ and $V$ are quadratic. The point being that such a  case can be solved explicitly and hence provides a guidance for what can be expected in the general case.

Let us consider and Hamiltonian system with four degree of freedom
$(q_1,p_1, q_2,p_2)\in \bR^8$ given by
\[
H_\ve:=\frac 12\{\|p_1\|^2+\|p_2\|^2\}+\frac
12\{\|q_1\|^2+\|q_2\|^2\}+\ve\|q_1-q_2\|^2 
\]
plus random forces that conserves the kinetic energy (that is
independent diffusions on the circles $\|p_i\|^2=cost$). To this end
consider the vector fields 
\[
X_i:=p_{i,1}\partial_{p_{i,2}}-p_{i,2}\partial_{p_{i,1}}.
\]
The generator is thus given by
\[
L_\ve:=\{H_\ve,\cdot\}+\sigma^2\sum_{i=1}^2X_i^2 = A_\ve + \sigma^2 S
\]
The energies of the two particles are $\mc E_i=\frac 12\|p_i\|^2+\frac
12\|q_i\|^2+\frac{\ve}2\|q_1-q_2\|^2$ and 
\[
\begin{split}
\partial_t \mc E_1&=\ve\langle q_2-q_1, p_1+p_2\rangle=:\ve j\\
\partial_t \mc E_2&=-\ve j
\end{split}
\]
gives the current.
A direct computation shows that, setting
\[
u:=\frac 12\{\|q_2\|^2-\|q_1\|^2\} - \sigma^{-2} q_2\cdot p_1 + \sigma^{-2}
q_1\cdot p_2 
\]
holds true
\begin{equation}\label{eq:poisson}
L_\ve u=j+\ve\sigma^{-2}\{\|q_1\|^2-\|q_2\|^2\}=:\tilde j.
\end{equation}
Accordingly, if we rescale the time by $\ve^{-2}$ and we look at the
random variables 
$\mc E_{i,\ve}(t):=\mc E_i(\ve^{-2}t)$.

\begin{equation}\label{calc1}
  \begin{split}
    \mc E_{1,\ve}(t) - \mc E_{1,\ve}(0) = \ve \int_0^{\ve^{-2} t} (L_\ve u)(s)
    ds - \ve^2\sigma^{-2} \int_0^{\ve^{-2}t} (\|q_1(s)\|^2 - \|q_2(s)\|^2) ds \\
    = \ve\left( u(\ve^{-2}t) - u(0) \right) + \ve M^u_{\ve^{-2}t} -
    \sigma^{-2} \int_0^{t} (\mc E_{1,\ve}(\tau) -\mc E_{2,\ve}(\tau)) d\tau \\
    - \ve^2\sigma^{-2} \int_0^{\ve^{-2}t} \big[ (\|q_1(s)\|^2-\|p_1(s)\|^2) -
    (\|q_2(s)\|^2 - \|p_2(s)\|^2)\big] ds +\cO(\ve^3)
  \end{split}
\end{equation}
where the martingale $M^u_t$ has quadratic variation given by
\begin{equation*}
  <M^u>_t = \sigma^2\int_0^t \left[ (X_1 u(s))^2 + (X_2 u (s))^2 \right] ds
\end{equation*}

We first show that the average of the last term on the RHS of
(\ref{calc1}) tends to  0. Observe that
\begin{equation*}
  \begin{split}
    (\|q_1\|^2-\|p_1\|^2)& - (\|q_2\|^2 - \|p_2\|^2) \\
    &= - L_\ve\left(p_1\cdot q_1 - p_2\cdot q_2 + \frac 12 (\|q_1\|^2 -
      \|q_2\|^2)\right) - \ve \left(\|q_1\|^2 - \|q_2\|^2\right)
  \end{split}
\end{equation*}
Calling $v = p_1\cdot q_1 - p_2\cdot q_2 + \frac 12 (\|q_1\|^2 -
\|q_2\|^2)$, the last term on the RHS of (\ref{calc1}) can be rewritten as
\begin{equation*}
  \ve^2\left( v(\ve^{-2} t) - v(0)\right) + \ve^2 M^v_{\ve^{-2}t} +
  \ve^3  \int_0^{\ve^{-2}t} (\|q_1(s)\|^2 - \|q_2(s)\|^2) ds
\end{equation*}
It is easy to show that the average goes to 0
as $\ve \to 0$.

It remains to compute the limit of the martingale $\ve
M^u_{\ve^{-2}t}$. To this purpose one has to compute the limit of
its quadratic variation
\begin{equation*}
   <\ve M^u>_{\ve^{-2}t} = \ve^2\sigma^{-2} \int_0^{\ve^{-2}t} \left[
     (q_{2,1}p_{1,2} - q_{2,2}p_{1,1})^2 +
     (q_{1,2}p_{2,1} - q_{1,1}q_{2,2})^2 \right] ds 
\end{equation*}
After some explicit calculation, one can find 4th order polynomials
$\psi,\phi, \chi$ such that 
\begin{equation*}
  (q_{2,1}p_{1,2} - q_{2,2}p_{1,1})^2 +
     (q_{1,2}p_{2,1} - q_{1,1}q_{2,2})^2 = 2 \mc E_1 \mc E_2 + L_\ve \psi + S
     \phi + \ve \chi
\end{equation*}
The term $L_\ve \psi +  \ve \chi$ does not contribute. For the term $S\phi$, by using Schwartz
inequality:
\begin{equation*}
  \begin{split}
    - \int_0^t d\tau \int S&\phi F_\tau d\mu_\ve= 
    \int_0^t d\tau \sum_{i=1,2}\int (X_i\phi) (X_iF_\tau) d\mu_\ve \\ 
   & \le \left(\int_0^t \mathcal D(F_\tau)  d\tau\right)^{1/2} 
    \left(\int_0^t \sum_{i=1,2}\int (X_i\phi)^2 F_\tau d\mu_\ve
      d\tau\right)^{1/2}\\
      &\leq   \left(\int_0^t \mathcal D(F_\tau)  d\tau\right)^{1/2} \left(\int_0^t \sum_{i=1,2}\|F\|_L^2\left[\int (X_i\phi)^4 d\mu_\ve\right]^{\frac 12}
      d\tau\right)^{1/2} \le C \sqrt t
  \end{split}
\end{equation*}

We have obtained that $\mc E_{1,\ve}, \mc E_{2,\ve}$ converge to the
(degenerate) diffusion on $R_+^2$ generated by
\begin{equation}
  \label{eq:1bis}
  \mathcal L = 2\sigma^{-2}(\partial_{\mc E_1} - \partial_{\mc E_2}) \mc E_1 \mc E_2
  (\partial_{\mc E_1} - \partial_{\mc E_2}) 
\end{equation}
corresponding to the stochastic differential equation:
\begin{equation}
  \label{eq:2bis}
  \begin{split}
    d\mc E_1 = \sigma^{-1}\sqrt{2 \mc E_1 \mc E_2}\; dw_t - 2\sigma^{-2}(\mc E_1 - \mc E_2)\; dt = -d\mc E_2
  \end{split}
\end{equation}
with $w_t$ standard Wiener process.


\begin{thebibliography}{999}
\footnotesize




\bibitem{BO} C. Bernardin, S. Olla, {\em Fourier's law for a microscopic model of heat conduction},  Journal of  Statistical Physics, vol.{\bf 118}, nos.3/4, 271-289, (2005).

\bibitem{BLR} F. Bonetto, J.L. Lebowitz, Rey-Bellet, {\em Fourier's law: A challenge to theorists}, Mathematical Physics 2000, Imperial College Press, London, 2000, pp.128-150.

\bibitem{cerrai} S. Cerrai, Ph. Cl\'ement, Well-posedness of the
  martingale problem for some degenerate diffusion processes occurring
  in dynamics of populations, Bull. Sci. Math.128 (2004) 355â--389. 

\bibitem{freidlin} Freidlin, M. I. {\em Fluctuations in dynamical
    systems with averaging} (Russian)  Dokl. Akad. Nauk SSSR  226
  (1976), no. 2, 273--276.

\bibitem{FW} Freidlin, M.I., Wentzell, A.D., {\em Random Perturbations of Dynamical Systems}, 2nd edn. Springer, Heidelberg (1998) 



\bibitem{horm} Hormander, L., \emph{Hypoelliptic second order
  differential equations}. Acta Math. 119, 147--171 (1967).

\bibitem{Ki} Kifer, Y. {\em Some recent advances in averaging. In: Modern Dynamical Systems and Applications}, pp. 385--403. Cambridge University Press, Cambridge (2004)

\bibitem{klo} Komorowski, Landim, Olla, \emph{Fluctuation in Markov
    Processes}, Book to appear (2010).

\bibitem {ovy} S. Olla, S. Varadhan, H. Yau,\emph{ Hydrodynamical
    limit for 
  a Hamiltonian system with weak noise}, Commun. Math. Phys. {\bf155}
  (1993), 523-560. 

\bibitem{RLL} Z. Rieder, J. L. Lebowitz, and E. Lieb, {\em Properties of a Harmonic Crystal in a Stationary Nonequilibrium State}, J. Math. Phys. {\bf 8}, 1073 (1967).


\bibitem{svy} S. Sethuraman, S.R.S. Varadhan, H.T. Yau,
\emph{Diffusive limit of a tagged particle in asymmetric simple
  exclusion processes}, Comm. Pure Appl. Math. 53 (2000), {\bf 8},
972--1006.  

\bibitem{Va} S.R.S. Varadhan, {\em Nonlinear diffusion limit for a
    system with nearest neighbor interactions-II},  
Asymptotic problems in probability theory: stochastic models and
diffusions on fractals (Sanda/Kyoto, 1990), 75--128, Pitman Res. Notes
Math. Ser., 283, Longman Sci. Tech., Harlow, 1993. 

\bibitem{V1} Cedric Villani, {\em Hypocoercivity}.
  Mem. Amer. Math. Soc.  202  (2009),  no. 950, 141 pp. 


\end{thebibliography}
\end{document}